\documentclass[11pt]{article}
\usepackage[dvips]{graphicx}
\usepackage[dvips,letterpaper,margin=1.5in]{geometry}
\usepackage{color}
\usepackage{amsmath}
\usepackage{amsfonts}
\usepackage{amssymb}
\usepackage{amsthm}
\usepackage{newlfont}
\usepackage{epsfig}
\usepackage{multirow}
\usepackage[version=4]{mhchem}
\usepackage{enumerate}
\usepackage{setspace}
\usepackage{cite}
\usepackage{tikz}
\usepackage{subfig}
\usetikzlibrary{shapes,arrows,automata}
\usepackage{colortbl}
\renewcommand{\emptyset}{\varnothing}
\usepackage[textsize=small]{todonotes}

\usepackage{appendix}

\begin{document}

\newtheorem{definition}{Definition}[section]
\newtheorem{lemma}{Lemma}[section]
\newtheorem{theorem}{Theorem}[section]
\newtheorem{corollary}{Corollary}[section]
\newtheorem{remark}{Remark}[section]
\newtheorem{example}{Example}[section]


\small

\title{Conditions for Extinction Events in Chemical Reaction Networks with Discrete State Spaces}
\author{Matthew D. Johnston$^{a}$, David F. Anderson$^{b}$,\\Gheorghe Craciun$^{b,c}$, and Robert Brijder$^{d}$\\ \\ ${}^{a}$ Department of Mathematics,\\San Jos\'{e} State University, San Jos\'{e}, CA, 95192 USA\\
${}^{b}$ Department of Mathematics,\\University of Wisconsin-Madison, Madison, WI, 53706 USA\\
${}^{c}$ Department of Biomolecular Chemistry,\\University of Wisconsin-Madison, Madison, WI, 53706 USA\\
 ${}^{d}$ Department WET-INF,\\ Hasselt University, Diepenbeek, Belgium}
\date{}
\maketitle


\vspace{0.2in}

\begin{abstract}
\small

We study chemical reaction networks with discrete state spaces, such as the standard continuous time Markov chain model, and present sufficient conditions on the structure of the network that guarantee the system exhibits an extinction event.
The conditions we derive involve creating a modified chemical reaction network called a domination-expanded reaction network and then checking properties of this network. We apply the results to several networks including an EnvZ-OmpR signaling pathway in \emph{Escherichia coli}. This analysis produces a system of equalities and inequalities which, in contrast to previous results on extinction events, allows algorithmic implementation. Such an implementation will be investigated in a companion paper where the results are applied to 458 models from the European Bioinformatics Institute's BioModels database.

\end{abstract}


\bigskip

\section{Introduction}
\label{introduction}

Continuous state differential equations  are a popular  modeling choice for the chemical concentrations of biochemical reaction networks in several disciplines, including industrial chemistry and systems biology. However, differential equations 
should only be used to model  chemical concentrations  when the counts of the reactant species are high \cite{Kurtz2,AndKurtz2015,AndKurtz2011}.  When the multiplicity of the individual species is low, as is often the case in enzymatic and genetic systems, it is important to use a model with a discrete state space which tracks individual molecular counts.

Predictions pertaining to the long-term behavior of a particular system can change dramatically depending upon whether the system is modeled with differential equations or with a discrete state space. In particular, discrete-space models may exhibit an extinction event where none exists in the corresponding continuous state model.
For example, consider the following chemical reaction network:

\begin{center}
\begin{tikzpicture}[auto, outer sep=3pt, node distance=2cm,>=latex']
\node (C1) {$2X_1$};
\node [right of = C1, node distance = 2.5cm] (C2) {$X_1+X_2$};
\node [right of = C2, node distance = 2.5cm] (C3) {$2X_2$};
\path[->,bend left=10] (C1) edge node [above left = -0.15cm] {\tiny $1$} (C2);
\path[->,bend left=10] (C2) edge node [below right = -0.15cm] {\tiny $2$} (C1);
\path[->] (C2) edge node [above = -0.15cm] {\tiny $3$} (C3);
\end{tikzpicture}
\end{center}

\noindent 
where the labels correspond to the enumeration of the reactions. The deterministic mass action model predicts an asymptotically stable steady state for a wide range of parameter values. However, for the discrete-space model with stochastic mass-action kinetics  and $M=X_1(0) + X_2(0)$, the state $\{ \mathbf{X}_1 = 0, \mathbf{X}_2 = M\}$   is the inevitable absorbing state  regardless of parameter values. This extinction event can be achieved by reaction 3 occurring until the count of species $X_1$ is zero, at which point no further reactions may occur.

Several frameworks exist for tracking trajectories of discrete state  chemical reaction systems, including those of continuous time Markov chains \cite{AndKurtz2015,AndKurtz2011} and stochastic Petri nets \cite{Bause-Kritzinger}. In these settings, the admissible transitions between states are assumed to occur randomly at a known  rate and the occurrence of each reaction instantaneously updates the system according to the stoichiometry of the associated reaction. Analysis of such systems is typically conducted by generating sample trajectories (through a stochastic simulation algorithm, e.g. \emph{Gillespie's Algorithm} \cite{Gillespie} or the next reaction method \cite{G-B,Anderson}), by analyzing the evolution of the probability distribution via Kolmogorov's forward equations (i.e.~the \emph{chemical master equation}), by characterizing the stationary distributions of the models \cite{A-C-K}, or by studying the stochastic equations for the model \cite{AndKurtz2011,AndKurtz2015}.

The study of extinction events in discrete interaction models is well-established in population dynamics and epidemic modeling, but the corresponding study in systems biology has only recently gained widespread attention.
 In \cite{A-E-J}  Anderson \emph{et al.}~described a large class of systems for which an extinction event necessarily occurs in the discrete model.  Interestingly, this class of models had previously been shown to have a particular ``robustness'' when modeled with deterministic ordinary differential equations \cite{Sh-F}.   In \cite{Brijder} R.~Brijder utilized tools from Petri Net Theory to further extend the scope of networks known to have extinction behavior, by relating a kernel condition introduced in \cite{A-E-J} to the $T$-invariants of the corresponding Petri net. Related recent work analyzing transient and post-extinction behavior in discrete chemical reaction systems can be found in \cite{Anderson2016,Enciso2016}. 


In this paper, we further develop a network-based approach to determining when discrete-space chemical reaction systems may exhibit an extinction event. Our main results, Theorem \ref{mainresult} and Corollary \ref{maincorollary}, state that a chemical reaction network with a discrete state space exhibits an extinction event if there is a modified network, called the domination-expanded reaction network, on which a particular set of inequalities on the edges \textit{cannot} be satisfied. The conditions we present may be summarized as systems of equalities and inequalities and, like Corollary 2 of \cite{Brijder}, suggests computational implementation. 
Such an implementation will be explored in further depth in a follow-up paper \cite{J2017}. We demonstrate the effectiveness of Corollary \ref{maincorollary} on several models, including a model of the EnvZ-OmpR signaling pathway in \emph{Escherichia coli} \cite{Sh-F}.

The notation of the paper is drawn from \emph{chemical reaction network theory} which has proven  effective for relating topological properties of a network's reaction graph to its admissible qualitative dynamical behaviors \cite{F1,H,H-J1,Fe2,F3,Fe4}. The notions introduced here may be equivalently defined in the context of Petri nets, which we summarize in Appendix \ref{appendixd}  \cite{Bause-Kritzinger,Brijder}. We also adopt the following common notation throughout the paper:
\begin{itemize}
\item
$\mathbb{R}_{\geq 0} = \{x \in \mathbb{R} \mid x \geq 0\}$ and $\mathbb{R}_{> 0} = \{x \in \mathbb{R} \mid x > 0\}$,
\item
for $\mathbf{v} = (v_1,\ldots,v_n) \in \mathbb{R}_{\geq 0}^n$, we define $\mathrm{supp}(\mathbf{v}) = \{ i \in \{1,\dots,n\} \; | \; v_i > 0 \}$,
\item
for a set $X = \{ X_1, X_2, \ldots, X_n \}$ of indexed elements and a subset $W \subseteq X$, we define $\mathrm{supp}(W) = \{ i \in \{1,\dots,n\} \; | \; X_i \in W \}$,
\item
for a subset $W \subseteq X$, we define the complement $W^c = \{ x \in X \; | \; x \not\in W \}$,
\item
for $\mathbf{v}, \mathbf{w} \in \mathbb{R}^n$, we define $\mathbf{v} \leq \mathbf{w}$ if $v_i \leq w_i$ for each $i \in \{1,\dots,n\}$.
\end{itemize}

\section{Background}
\label{background}

We outline the background notation and terminology relevant to the study of \emph{chemical reaction network theory} (CRNT). (For further background, see Martin Feinberg's online lecture notes \cite{F3}.)

\subsection{Chemical Reaction Networks}
\label{crntsection}


The fundamental object of interest in CRNT is the following.

\begin{definition}
\label{crn}
A \textbf{chemical reaction network} (CRN) is given by a triple of finite sets $(\mathcal{S},\mathcal{C},\mathcal{R})$ where:
\begin{enumerate}
\item
The \textbf{species set} $\mathcal{S} = \{ X_1, \ldots, X_m \}$ contains the species of the CRN.
\item
The \textbf{reaction set} $\mathcal{R} = \{ R_1, \ldots, R_r \}$ consists of ordered pairs $(y,y') \in \mathcal{R}$ where
\begin{equation}
\label{complex}
y = \sum_{i=1}^m y_i X_i \; \mbox{ and } \; y' = \sum_{i=1}^m y_i' X_i,
\end{equation}
and where the values $y_{i},y_{i}' \in \mathbb{Z}_{\geq 0}$ are the \textbf{stoichiometric coefficients}. We will also write reactions $(y,y')$ as $y \to y'$.
\item The \textbf{complex set} $\mathcal{C}$ consists of the linear combinations of the species in \eqref{complex}.  Specifically,
 $\mathcal{C} = \{ y \,  | \, y\to y' \in \mathcal{R} \} \cup \{ y' \, | \, y\to y' \in \mathcal{R} \}$. The number of  distinct complexes is denoted $|\mathcal{C}| = n$.

Allowing for a slight abuse of notation, we will let $y$ denote both the complex itself and the complex vector $y = (y_1,\ldots,y_m)^T \in \mathbb{Z}_{\geq 0}^m$.
\end{enumerate}
\end{definition}



\noindent We assume an arbitrary but fixed ordering of the species, reactions and complexes. It is common to impose that a CRN does not contain any self-loops (i.e.\ reactions of the form $y \to y$). Since this assumption is not used in our results, and since it is common to allow self-loops in Petri Net Theory, we will not make this assumption here.

The interpretation of reactions as directed edges naturally gives rise to a reaction graph $G=(V,E)$ where the set of vertices is given by the complexes (i.e.~$V = \mathcal{C}$) and the set of edges is given by the reactions (i.e.~$E = \mathcal{R})$. The following terminology will be used.

\begin{enumerate}[(i)]

\item A complex $y$ is \emph{connected} to a complex $y'$ if there exists a sequence of complexes $y = y_{\mu(1)}, y_{\mu(2)}, \ldots, y_{\mu(\ell)} = y'$ such that either $y_{\mu(k)} \to y_{\mu(k+1)}$ or $y_{\mu(k+1)} \to y_{\mu(k)}$ for all $k \in \{1,\dots,\ell-1\}$.

\item There is a \emph{path} from $y$ to $y'$ if there is a sequence of distinct complexes such that $y = y_{\mu(1)} \to y_{\mu(2)} \to \cdots \to y_{\mu(\ell)} = y'$.

\item A maximal set of mutually connected complexes is called a \emph{linkage class} (LC) while a maximal set of mutually path-connected complexes is called a \emph{strong linkage class} (SLC). 
The set of linkage classes will be denoted $\mathcal{L}$ 
while the set of SLCs will be denoted $\mathcal{W}$. 

\item
An SLC $W \in \mathcal{W}$ is called \emph{terminal} if there are no outgoing reactions, i.e.~$y\in W$ and $y\to y'\in \mathcal{R}$ implies $y'\in W$. The set of terminal SLCs will be denoted $\mathcal{T} \subseteq \mathcal{W}$. A complex $y \in \mathcal{C}$ is called terminal if it belongs to a terminal SLC, and a reaction $y \to y' \in \mathcal{R}$ is terminal if $y$ is terminal. 
\item
A set $\mathcal{Y} \subseteq \mathcal{C}$ is called an \emph{absorbing complex set} if it contains every terminal complex and has no outgoing edges, i.e. $y \in \mathcal{Y}$ and $y \to y' \in \mathcal{R}$ implies $y' \in \mathcal{Y}$. A complex $y \in \mathcal{Y}$ is called $\mathcal{Y}$-interior, and a reaction $y \to y' \in \mathcal{R}$ is called $\mathcal{Y}$-interior if $y$ is $\mathcal{Y}$-interior; otherwise they are $\mathcal{Y}$-exterior.
\end{enumerate}
\noindent Absorbing complex sets are a generalization of the set of terminal complexes of a CRN, since they must contain, but may be strictly larger than, this set. 
Note that the set of terminal complexes is a closed complex set of the CRN, as is the set $\mathcal{Y} = \mathcal{C}$. We will be particularly interested in the case where $\mathcal{Y}$ is the set of terminal complexes, as this provides the foundation upon which our main results are built.

To each reaction $y \to y' \in \mathcal{R}$ we associate a \emph{reaction vector} $y' - y \in \mathbb{Z}^m$ which tracks the net gain and loss of each chemical species as a result of the  occurrence of this reaction. The \emph{stoichiometric subspace} is defined by
\[S = \mbox{span} \left\{ y' - y \in \mathbb{Z}^m \; | \; y \to y' \in \mathcal{R} \right\}.\]
\noindent The \emph{stoichiometric matrix} $\Gamma \in \mathbb{Z}^{m \times r}$ is the matrix with the reaction vectors as columns. 

A CRN is said to be \emph{conservative} (respectively, \emph{subconservative}) if there exists a $\mathbf{c} \in \mathbb{R}_{> 0}^m$ such that $\mathbf{c}^T \Gamma = \mathbf{0}^T$ (respectively, $\mathbf{c}^T \Gamma \leq \mathbf{0}^T$). The vector $\mathbf{c}$ is called a \textit{conservation vector}.  Conservative CRNs have the property that a particular linear combination of all species remains constant as a result of each reaction, while subconservative CRNs have a combination of species which is nonincreasing in every reaction. A common example is conservation in the overall amount of enzyme or substrates in a closed enzymatic system, but there need not be such a physical interpretation.

We present three examples in order to illustrate definitions.
\begin{example}
\label{example234}
Consider the following CRN:
\begin{center}
\begin{tikzpicture}[auto, outer sep=3pt, node distance=2cm,>=latex']
\node (C1) {$X_1+X_2$};
\node [right of = C1, node distance = 2.5cm] (C2) {$2X_2$};
\node [below of = C1, node distance = 1cm] (ghost) {$$};
\node [right of = ghost, node distance = 0.44cm] (C3) {$X_2$};
\node [below of = C2, node distance = 1cm] (C4) {$X_1$};
\path[->, bend left = 10] (C1) edge node [above left = -0.15cm] {\tiny $1$} (C2);
\path[->, bend left = 10] (C2) edge node [below right = -0.15cm] {\tiny $2$} (C1);
\path[->] (C3) edge node [above = -0.15cm] {\tiny $3$} (C4);
\end{tikzpicture}
\end{center}
This CRN has the sets $\mathcal{S} = \{ X_1, X_2\}$, $\mathcal{R} = \{ X_1 + X_2 \to 2X_2, 2X_2 \to X_1 + X_2, X_2 \to X_1 \}$, and $\mathcal{C} = \{ X_1 + X_2, 2X_2, X_2, X_1 \}$. The linkage classes are
\[
	\mathcal{L} = \left\{ \{X_1 + X_2, 2X_2 \}, \{ X_2, X_1 \} \right\}
	\]
	 while the SLCs are
	 \[
	 	\mathcal{W} = \left\{ \{ X_1 + X_2, 2X_2 \}, \{X_2 \}, \{X_1 \}\right\}.
	\]
	 Note that SLCs may consist of singletons. The terminal SLCs are
	 \[
	 \mathcal{T} = \left\{ \{ X_1 + X_2, 2X_2 \}, \{ X_1 \} \right\}.
	\]
 The stoichiometric matrix is as follows:
\[\Gamma = \left[ \begin{array}{ccc} -1 & 1 & -1 \\ 1 & -1 & 1 \end{array} \right]. \]
\noindent The stoichiometric subspace is given by $S = \mbox{span} \{ (1,-1)^T \}$, and there is the conservation vector $\mathbf{c} = (1,1)^T$. This conservation vector represents the fact that $X_1 + X_2$ is constant.
\end{example}

\begin{example}
\label{example235}
Consider the following CRN:
\begin{center}
\begin{tikzpicture}[auto, outer sep=3pt, node distance=2cm,>=latex']
\node (C1) {$X_1$};
\node [right of = C1, node distance = 2cm] (C2) {$2X_2$};
\node [below of = C1, node distance = 0.75cm] (C3) {$X_2$};
\node [right of = C3, node distance = 2cm] (C4) {$2X_1$};
\path[->] (C1) edge node [above = -0.15cm] {\tiny $1$} (C2);
\path[->] (C3) edge node [above = -0.15cm] {\tiny $2$} (C4);
\end{tikzpicture}
\end{center}
The set of terminal complexes is $\{ 2X_2, 2X_1 \}$. There are several additional choices for absorbing complex sets, including $\mathcal{Y} = \{ X_1, 2X_2, 2X_1\}$ and $\mathcal{Y} = \{ 2X_2, X_2, 2X_1 \}$. The stoichiometric matrix is as follows:
\[\Gamma = \left[ \begin{array}{cc} -1 & 2 \\ 2 & -1 \end{array} \right]. \]
\noindent The stoichiometric subspace is given by $S = \mbox{span} \{ (-1, 2)^T, (2, -1)^T \} = \mathbb{R}^2$. There is no vector $\mathbf{c} \in \mathbb{R}_{> 0}^2$ for which $\mathbf{c}^T \Gamma \le \mathbf{0}^T$, so the CRN is not conservative or subconservative.

\end{example}

\begin{example}
\label{example236}
Consider the following CRN:
\begin{center}
\begin{tikzpicture}[auto, outer sep=3pt, node distance=2cm,>=latex']
\node (C1) {$X_1+X_2$};
\node [right of = C1, node distance = 2cm] (C2) {$X_1$};
\node [right of = C2, node distance = 2cm] (C3) {$X_2$};
\path[->] (C1) edge node [above = -0.15cm] {\tiny $1$} (C2);
\path[->, bend left = 10] (C2) edge node [above left = -0.15cm] {\tiny $2$} (C3);
\path[->, bend left = 10] (C3) edge node [below right = -0.15cm] {\tiny $3$} (C2);
\end{tikzpicture}
\end{center}
The stoichiometric matrix is as follows:
\[\Gamma = \left[ \begin{array}{ccc} 0 & -1 & 1 \\ -1 & 1 & -1 \end{array} \right]. \]
\noindent There is no vector $\mathbf{c} \in \mathbb{R}_{> 0}^2$ such that $\mathbf{c}^T \Gamma = \mathbf{0}^T$, so the CRN is not conservative; however, the vector $\mathbf{c} = (1,1)^T$ has the property that $\mathbf{c}^T \Gamma = (-1,0,0)\leq \mathbf{0}$ so that the CRN is subconservative.
\end{example}

\subsection{Chemical Reaction Networks with Discrete State Spaces}
\label{stochasticsection}



A discrete \emph{state} $\mathbf{X}$ is an element of $\mathbb{Z}_{\geq 0}^m$ and denotes the molecular counts of each species. We let $\mathbf{X}(t) = (\mathbf{X}_1(t), \ldots, \mathbf{X}_m(t))^T \in \mathbb{Z}_{\geq 0}^m$ denote the state where $\mathbf{X}_i(t)$ corresponds to the count of species $X_i$ at time $t$. These discrete states evolve as follows:
\begin{equation}
\label{markov}
\mathbf{X}(t) = \mathbf{X}(0) + \Gamma \; \mathbf{N}(t)
\end{equation}
where $\mathbf{N}(t) = (N_1(t),\ldots,N_r(t))^T$ and, for all $k \in \{1,\ldots,r\}$, $N_{k}(t) \in \mathbb{Z}_{\geq 0}$ is the number of times the $k$th reaction has occurred up to time $t$. There are several established frameworks for modeling the time-evolution of CRNs on discrete state spaces, including that of continuous time Markov chains (CTMCs) and stochastic Petri nets. We will not concern ourselves with precise dynamical details; rather, we will focus on where trajectories may evolve in $\mathbb{Z}_{\geq 0}^m$. For a similar treatment, see the paper of L. Paulev\'{e} \emph{et al.} \cite{P-C-K}.

We will say that a complex $y\in \mathcal{C}$ is \emph{charged} at state $\mathbf{X} \in \mathbb{Z}_{\ge 0}^m$ if $\mathbf{X}_i \geq y_i$ for all $i\in \{1,\dots,m\}$.  We will then say that reaction $y \to y' \in \mathcal{R}$ is \emph{charged} at state $\mathbf{X} \in \mathbb{Z}_{\ge 0}^m$ if the ``source complex'' $y$ is charged at $\mathbf{X}$.  Note that a reaction is therefore charged at a state $\mathbf{X}$ if the species counts are sufficient for the source complex of that reaction.



We will be primarily interested in how trajectories $\mathbf{X}(t)$ move through the state space $\mathbb{Z}_{\geq 0}^m$ of subconservative CRNs. In particular, we will be interested in the long-term behavior. We therefore introduce the following terminology, which is  adapted from the conventions of  stochastic processes.
\begin{definition}
\label{recurrence}
Consider a CRN on a discrete state space. Then:
\begin{enumerate}
\item
A state $\mathbf{X} \in \mathbb{Z}_{\geq 0}^m$ \textbf{reacts to} a state $\mathbf{Y} \in \mathbb{Z}_{\geq 0}^m$ (denoted $\mathbf{X} \to \mathbf{Y}$) if there is a reaction $y \to y' \in \mathcal{R}$ such that $\mathbf{Y} = \mathbf{X} + y' - y$ and $y$ is charged at state $\mathbf{X}$.
\item
A state $\mathbf{Y} \in \mathbb{Z}_{\geq 0}^m$ is \textbf{reachable} from a state $\mathbf{X} \in \mathbb{Z}_{\geq 0}^m$ (denoted $\mathbf{X} \leadsto \mathbf{Y})$ if there exists a sequence of states such that $\mathbf{X} = \mathbf{X}_{\nu(1)} \to \mathbf{X}_{\nu(2)} \to \cdots \to \mathbf{X}_{\nu(l)} = \mathbf{Y}$.
\item
A state $\mathbf{X} \in \mathbb{Z}_{\geq 0}^m$ is \textbf{recurrent} if, for any $\mathbf{Y} \in \mathbb{Z}_{\geq 0}^m$, $\mathbf{X} \leadsto \mathbf{Y}$ implies $\mathbf{Y} \leadsto \mathbf{X}$; otherwise, the state is \textbf{transient}.
\end{enumerate}
\end{definition}

\noindent Note that the state space of a subconservative CRN is finite (see Theorem 1, \cite{DBLP:conf/ac/MemmiR75}). For this classification of CRNs, therefore, the notion of recurrence introduced above therefore agrees with the notion of positive recurrence from the language of CTMC (see \cite{Lawler}).



We now extend the properties of recurrence and transience of states to the complexes and reactions of a CRN. Further considerations on the recurrence properties of the SLCs of a CRN are contained in Appendix \ref{appendixc}.

\begin{definition}
\label{turnedon}
Consider a CRN on a discrete state space. Then:
\begin{enumerate}
\item
A complex $y \in \mathcal{C}$ is \textbf{recurrent} from state $\mathbf{X} \in \mathbb{Z}_{\geq 0}^m$ if $\mathbf{X} \leadsto \mathbf{Y}$ 
implies that there is a $\mathbf{Z}$ for which $\mathbf{Y} \leadsto \mathbf{Z}$
and $y$ is charged at $\mathbf{Z}$; otherwise, $y$ is \textbf{transient} from $\mathbf{X}$.
\item
A reaction $y \to y' \in \mathcal{R}$ is \textbf{recurrent} from state $\mathbf{X} \in \mathbb{Z}_{\geq 0}^m$ if the source complex $y$ is recurrent from $\mathbf{X}$; otherwise, $y \to y' \in \mathcal{R}$ is \textbf{transient} from $\mathbf{X}$.
\end{enumerate}
\end{definition}


\noindent In plain English, a complex $y$ is recurrent from a state $\mathbf{X}$ if, whenever the process can go from the state $\mathbf{X}$ to the state $\mathbf{Y}$, then the process can move from the state $\mathbf{Y}$ to some state $\mathbf{Z}$ where $y$ is charged.

The following clarifies the type of behavior for CRNs on discrete state spaces in which we will be interested.

\begin{definition}
	Consider a CRN on a discrete state space.
	We will say that the CRN exhibits:
\begin{enumerate}
\item
an \textbf{extinction event on $\mathcal{Y} \subseteq \mathcal{C}$ from $\mathbf{X}\in \mathbb{Z}^m_{\ge 0}$} if every complex $y \in \mathcal{Y}$ is transient from $\mathbf{X}$.
\item
a \textbf{guaranteed extinction event on $\mathcal{Y} \subseteq \mathcal{C}$} if it has an extinction event on $\mathcal{Y}$ from every $\mathbf{X}\in \mathbb{Z}^m_{\ge 0}$.
\end{enumerate}
\end{definition}


\begin{example}
Consider the CRN in Example \ref{example234}. Through repeated application of reaction $3$, we can arrive at the state $\{ \mathbf{X}_1 = M, \mathbf{X}_2 = 0 \}$ where $M = \mathbf{X}_1(0) + \mathbf{X}_2(0)$. Since this is a possible outcome from any initial $\mathbf{X} \in \mathbb{Z}_{\ge 0}^2$, we have that this CRN has a guaranteed extinction event on $\mathcal{Y} = \{X_1+X_2, 2X_2, X_2\}$. Notice that no reaction may occur after the extinction event.
\end{example}

\begin{example}
	Consider the CRN in Example \ref{example236}.
Notice that the reaction $X_1+X_2 \to X_1$ cannot occur indefinitely since all other reactions in the CRN preserve $\mathbf{X}_1(t) + \mathbf{X}_2(t)$.     It follows that the model has a guaranteed extinction event on $\mathcal{Y} = \{ X_1+X_2 \}$. Notice, however, that so long as $\mathbf{X}_1(0) + \mathbf{X}_2(0) \ge 1$ the reactions $X_1 \to X_2$ and $X_2 \to X_1$ are both recurrent. An extinction event therefore does not necessarily imply that all reactions must cease.
\end{example}

\section{Main results}
\label{extinctionsection}

In this section, we motivate and present the main new constructions and theory of the paper (Theorem \ref{mainresult} and Corollary \ref{maincorollary}).


\subsection{Domination-expanded Reaction Networks}
\label{dominationsection}

We introduce the following.

\begin{definition}
\label{complexdomination}
Let $y, y' \in \mathcal{C}$ denote two distinct complexes of a CRN. We say that $y$ \textbf{dominates} $y'$ if $y' \leq y$. We define the \textbf{domination set} of a CRN to be
\begin{equation}
\label{dominationset}
\mathcal{D}^* = \left\{ (y,y') \in \mathcal{C} \times \mathcal{C} \; | \; y' \leq y, \; y \not= y' \right\}.
\end{equation}
\end{definition}

The notion of complex domination was introduced by D. Anderson \emph{et al.}\ in \cite{A-E-J} as an adaptation of the notion of ``differing in one species'' introduced by G. Shinar and M. Feinberg in \cite{Sh-F}. The domination property was extended to SLCs by R. Brijder in \cite{Brijder} where it was also shown that, for conservative CRNs, the domination properties give rise to a binary relation on the SLCs of a CRN whose transitive closure is a partial ordering on the SLCs of the CRN (Lemma 2, \cite{Brijder}). We consider further properties of transience and recurrence of SLCs in Appendix \ref{appendixc}. We note that the definition of complex domination in Definition \ref{complexdomination} is consistent with \cite{Brijder} but reversed from \cite{A-E-J}. 

\begin{example}
\label{example919}
Consider the CRNs from Examples \ref{example234}, \ref{example235}, and \ref{example236} respectively. 
For the CRN in Example \ref{example234}, we set $y_1 = X_1 + X_2$, $y_2 = 2X_2$, $y_3 = X_2$, and $y_4 = X_1$ and have $y_3 \leq y_1$, $y_3 \leq y_2$, and $y_4 \leq y_1$. For the CRN in Example \ref{example235}, we set $y_1 = X_1$, $y_2 = 2X_2$, $y_3 = X_2$, and $y_4 = 2X_1$, and have $y_1 \leq y_4$ and $y_3 \leq y_2$. For the CRN in Example \ref{example236}, we set $y_1 = X_1 + X_2$, $y_2 = X_1$, and $y_3 = X_2$, and have $y_2 \leq y_1$ and $y_3 \leq y_1$.
\end{example}

The key construction of this paper is the following, which uses the domination relations $\leq$ to expand CRNs into larger CRNs we call \emph{domination-expanded reaction networks}.

\begin{definition}
\label{dominationnetwork}
We say that $(\mathcal{S},\mathcal{C},\mathcal{R} \cup \mathcal{D})$ is a \textbf{domination-expanded reaction network} (dom-CRN) of the CRN $(\mathcal{S},\mathcal{C},\mathcal{R})$ if $\mathcal{D} \subseteq \mathcal{D}^*$. Furthermore, we say a dom-CRN is \textbf{$\mathcal{Y}$-admissible} if, given an absorbing complex set $\mathcal{Y} \subseteq \mathcal{C}$ of the dom-CRN, we have: (i) $\mathcal{R} \cap \mathcal{D} = \emptyset$, and (ii) $(y, y') \in \mathcal{D}$ implies $y' \not\in \mathcal{Y}$.
\end{definition}


\noindent A dom-CRN consists of the original CRN with additional directed edges corresponding to some (potentially all) of the domination relations $y' \leq y$. Note that the reaction arrows flow from the dominating complex to the ``smaller'' complex in the domination relation, i.e. $y' \leq y$ implies we add $y \to y'$. Consequently, like reactions, we will denote domination relations 
as either $(y,y')$ or $y \to y'$. A dom-CRN is admissible if we do not add any reactions which lead to the absorbing complex set $\mathcal{Y}$ of the dom-CRN. 

\begin{remark}
When applying Definition \ref{dominationnetwork}, we will commonly let the absorbing complex set $\mathcal{Y}$ coincide with the set of terminal complexes of the dom-CRN. In such cases, we will say a dom-CRN is simply \emph{admissible} with the understanding that $\mathcal{Y}$ is the set of terminal complexes.
\end{remark}

Note that a dom-CRN is a CRN in itself and therefore has associated to it all of the quantities and structural matrices given Section \ref{crntsection}. While a dom-CRN in general may have different structural properties than the original CRN, an important restriction is given by the following result, which is based on Lemma 2 of \cite{Brijder}. The proof is contained in Appendix \ref{appendixa}.

\begin{lemma}
\label{terminallemma}
If a CRN is subconservative, then for any dom-CRN: (i) the SLCs of the CRN and the dom-CRN coincide, and (ii) every terminal SLC of the dom-CRN is a terminal SLC of the CRN.
\end{lemma}

\noindent We can interpret Lemma \ref{terminallemma} as saying that, for a subconservative CRN, the addition of domination edges does not create new cycles between SLCs since this would create new SLCs. 

\begin{example}
\label{example920}
Consider the CRN from Examples \ref{example234} and \ref{example919}. Recall that the CRN is conservative, and therefore subconservative, so that Lemma \ref{terminallemma} applies. The maximal dom-CRN is given by the following:
\begin{center}
\begin{tikzpicture}[auto, outer sep=3pt, node distance=2cm,>=latex']
\node (C1) {$X_1+X_2$};
\node [right of = C1, node distance = 2.5cm] (C2) {$2X_2$};
\node [below of = C2, node distance = 1.5cm] (C3) {$X_2$};
\node [left of = C3, node distance = 2.5cm] (C4) {$X_1$};
\path[->, bend left = 10] (C1) edge node [above left = -0.15cm] {\tiny $1$} (C2);
\path[->, bend left = 10] (C2) edge node [below right = -0.15cm] {\tiny $2$} (C1);
\path[->] (C3) edge node [above = -0.15cm] {\tiny $3$} (C4);
\path[dashed,->] (C1) edge [right] node {\tiny $D_1$} (C3);
\path[dashed,->] (C1) edge [left] node {\tiny $D_2$} (C4);
\path[dashed,->] (C2) edge [right] node {\tiny $D_3$} (C3);
\end{tikzpicture}
\end{center}
where we have indexed the domination relations for clarity. As guaranteed by Lemma \ref{terminallemma}, the SLCs of the CRN and dom-CRN coincide. Notice that the terminal complex $X_1$ in the dom-CRN above is terminal in the original CRN, but that the terminal complexes $X_1 + X_2$ and $2X_2$ in the CRN are not terminal in the dom-CRN.

Notice also that this dom-CRN is not admissible since the domination relations 
$X_1 + X_2 \to X_1$ leads to the terminal complex $X_1$. Consider instead the subset $\mathcal{D} = \{ X_1 + X_2 \to X_2, 2X_2 \to X_2 \} \subset \mathcal{D}^*$ which corresponds to the following dom-CRN:
\begin{center}
\begin{tikzpicture}[auto, outer sep=3pt, node distance=2cm,>=latex']
\node (C1) {$X_1+X_2$};
\node [right of = C1, node distance = 2.5cm] (C2) {$2X_2$};
\node [below of = C2, node distance = 1.5cm] (C3) {$X_2$};
\node [left of = C3, node distance = 2.5cm] (C4) {$X_1$};
\path[->, bend left = 10] (C1) edge node [above left = -0.15cm] {\tiny $1$} (C2);
\path[->, bend left = 10] (C2) edge node [below right = -0.15cm] {\tiny $2$} (C1);
\path[->] (C3) edge node [above = -0.15cm] {\tiny $3$} (C4);
\path[dashed,->] (C1) edge [right] node {\tiny $D_1$} (C3);
\path[dashed,->] (C2) edge [right] node {\tiny $D_2$} (C3);
\end{tikzpicture}
\end{center}
This dom-CRN is admissible since $\mathcal{D}$ contains no domination edges which lead to the terminal complex $X_1$. 
\end{example}


\begin{example}
Consider the CRN from Example \ref{example235} and \ref{example919}. Recall that the CRN is neither conservative nor subconservative. Thus, Lemma \ref{terminallemma} stands silent.  The maximal dom-CRN is given by the following:
\begin{center}
\begin{tikzpicture}[auto, outer sep=3pt, node distance=2cm,>=latex']
\node (C1) {$X_1$};
\node [right of = C1, node distance = 2cm] (C2) {$2X_2$};
\node [below of = C2, node distance = 1.5cm] (C3) {$X_2$};
\node [left of = C3, node distance = 2cm] (C4) {$2X_1$};
\path[->] (C1) edge node [above = -0.15cm] {\tiny $1$} (C2);
\path[->] (C3) edge node [above = -0.15cm] {\tiny $2$} (C4);
\path[dashed,->] (C2) edge [right] node {\tiny $D$} (C3);
\path[dashed,->] (C4) edge [left] node {\tiny $D$} (C1);
\end{tikzpicture}
\end{center}
We have that there is only one SLC in the dom-CRN, which is given by $\{ X_1, 2X_2, X_2, 2X_1 \}$, so that the SLCs of the CRN and dom-CRN do not coincide.
We can see, therefore, that Lemma \ref{terminallemma}  does not hold in general if we remove the subconservative assumption.
\end{example}

\subsection{$\mathcal{Y}$-Exterior Forests and Balancing Vectors}

The following concept is adapted from numerous sources in graph theory. Trees have been used extensively in CRNT \cite{J1,C-D-S-S} and the related notion of arborescences factored in \cite{Boros2013-1}.

\begin{definition}
\label{forest1}
Consider a CRN $(\mathcal{S},\mathcal{C},\mathcal{R})$ and a $\mathcal{Y}$-admissible dom-CRN $(\mathcal{S},\mathcal{C},\mathcal{R} \cup \mathcal{D})$ where $\mathcal{Y} \subseteq \mathcal{C}$ is an absorbing complex set on the dom-CRN. Then $(\mathcal{S},\mathcal{C},\mathcal{R}_F \cup \mathcal{D}_F)$ where $\mathcal{R}_F \subseteq \mathcal{R}$ and $\mathcal{D}_F \subseteq \mathcal{D}$ is called an \textbf{$\mathcal{Y}$-exterior forest} if, for every complex $y \not\in \mathcal{Y}$, there is a unique path in $\mathcal{R}_F \cup \mathcal{D}_F$ from $y$ to $\mathcal{Y}$.
\end{definition}


\noindent A $\mathcal{Y}$-exterior forest is a forest in the usual sense in graph theory after restricting to the $\mathcal{Y}$-exterior portion of the reaction graph of the dom-CRN. Note that Definition \ref{forest1} places no restrictions on $\mathcal{Y}$-interior reactions. By convention, we will include such reactions in every $\mathcal{Y}$-exterior forest.
If $\mathcal{Y}$ consists solely of the terminal complexes of the dom-CRN, we say $(\mathcal{S},\mathcal{C},\mathcal{R}_F \cup \mathcal{D}_F)$ is simply an \emph{exterior forest}.

We will be interested in particular in $\mathcal{Y}$-exterior forests which satisfy the following property.

\begin{definition}
\label{balancing}
Consider a CRN $(\mathcal{S},\mathcal{C},\mathcal{R})$ and a $\mathcal{Y}$-admissible dom-CRN $(\mathcal{S},\mathcal{C},\mathcal{R} \cup \mathcal{D})$ where $\mathcal{Y} \subseteq \mathcal{C}$ is an absorbing complex set on the dom-CRN. Let $d = |\mathcal{D}|$. Then a $\mathcal{Y}$-exterior forest $(\mathcal{S},\mathcal{C},\mathcal{R}_F \cup \mathcal{D}_F)$ is said to be \textbf{balanced} if there is a vector $\alpha = (\alpha_R,\alpha_D) \in \mathbb{Z}_{\geq 0}^{r+d}$ with $\alpha_k>0$ for at least one $\mathcal{Y}$-exterior reaction which satisfies:
\begin{enumerate}
\item
supp$(\alpha_R) \subseteq$ supp$(\mathcal{R}_F)$ and supp$(\alpha_D) \subseteq$ supp$(\mathcal{D}_F)$;
\item
$\alpha_R \in \ker(\Gamma)$; and
\item
for every $R_k = y \to y' \in \mathcal{R}_F \cup \mathcal{D}_F$ where $y \not\in \mathcal{Y}$, we have $\displaystyle{\alpha_k \geq \sum_{R_l \in \Theta(y)} \alpha_l}$ where $\Theta(y) = \{ R_l \in \mathcal{R}_F \cup \mathcal{D}_F \; | \; R_l = y'' \to y\}$.
\end{enumerate}
Otherwise, the $\mathcal{Y}$-exterior forest is said to be \textbf{unbalanced}.
\end{definition}

\noindent The third condition of Definition \ref{balancing} can be interpreted as saying that, for every $y \not\in \mathcal{Y}$, the weight of the outgoing edge in the $\mathcal{Y}$-exterior forest must be at least as large as the sum of all incoming edges. When taken together, the three conditions of Definition \ref{balancing} generate a set of equalities and inequalities on the edges of the dom-CRN. This suggests a computational implementation, which is investigated in the companion paper \cite{J2017}.


\begin{example}
\label{example921}
Recall the CRN taken from Examples \ref{example234}, \ref{example919}, and \ref{example920} and the admissible dom-CRN from Example \ref{example920}. This dom-CRN admits several exterior forests, for example the following substructures in bold red:
\begin{center}
\begin{tikzpicture}[auto, outer sep=3pt, node distance=2cm,>=latex']
\node (C11) {$X_1+X_2$};
\node [right of = C11, node distance = 2.5cm] (C21) {$2X_2$};
\node [below of = C21, node distance = 1.5cm] (C31) {$X_2$};
\node [left of = C31, node distance = 2.5cm] (C41) {$X_1$};
\node [right of = C21, node distance = 2.5cm] (C12) {$X_1+X_2$};
\node [right of = C12, node distance = 2.5cm] (C22) {$2X_2$};
\node [below of = C22, node distance = 1.5cm] (C32) {$X_2$};
\node [left of = C32, node distance = 2.5cm] (C42) {$X_1$};
\path[red,line width=0.75mm,->, bend left = 10] (C11) edge node [above left = -0.15cm] {\tiny $\mathbf{1}$} (C21);
\path[->, bend left = 10] (C21) edge node [below right = -0.15cm] {\tiny $2$} (C11);
\path[red,line width=0.75mm,->] (C31) edge node [above = -0.15cm] {\tiny $\mathbf{3}$} (C41);
\path[dashed,->] (C11) edge [right] node {\tiny $D_1$} (C31);
\path[red,line width=0.75mm,dashed,->] (C21) edge [right] node {\tiny $\mathbf{D_2}$} (C31);
\path[->, bend left = 10] (C12) edge node [above left = -0.15cm] {\tiny $1$} (C22);
\path[red,line width=0.75mm,->, bend left = 10] (C22) edge node [below right = -0.15cm] {\tiny $\mathbf{2}$} (C12);
\path[red,line width=0.75mm,->] (C32) edge node [above = -0.15cm] {\tiny $\mathbf{3}$} (C42);
\path[red,line width=0.75mm,dashed,->] (C12) edge [right] node {\tiny $ \mathbf{D_1}$} (C32);
\path[dashed,->] (C22) edge [right] node {\tiny $D_2$} (C32);
\end{tikzpicture}
\end{center}
Note that every nonterminal complex has a unique path to $X_1$. We now check whether these exterior forests are balanced by Definition \ref{balancing} by checking equalities and inequalities on the vector of edges of the following form:
\[\begin{array}{r} \mbox{\emph{reaction:}} \\ \hline \alpha = \end{array}\hspace{-0.025in}\underbrace{\begin{array}{lll} \; 1 & 2 & 3 \\ \hline ((\alpha_R)_1, & (\alpha_R)_2, & (\alpha_R)_3 , \end{array} }_{\alpha_R}\hspace{-0.025in}\underbrace{\begin{array}{ll} D_1 & D_2 \\ \hline \; (\alpha_D)_1, & \; (\alpha_D)_2).\end{array} }_{\alpha_D}\]
Note also that the stoichiometric matrix is given by
\[\Gamma = \left[ \begin{array}{ccc} -1 & 1 & 1 \\ 1 & -1 & -1 \end{array} \right].\]

\begin{enumerate}
\item
In order for the left exterior forest to be balanced, it is required that we find a vector $\alpha = ((\alpha_R)_1, (\alpha_R)_2, (\alpha_R)_3, (\alpha_D)_1, (\alpha_D)_2) \in \mathbb{R}_{\geq 0}^5$, $\alpha \not= \mathbf{0}$, satisfying:
\[\left\{ \; \; \begin{aligned} (\mbox{Cond. } 1): \; \; & \hspace{0.17in} (\alpha_R)_2 = 0, \; (\alpha_D)_1 = 0 \\ (\mbox{Cond. } 2): \; \; & -(\alpha_R)_1 + (\alpha_R)_2 + (\alpha_R)_3 = 0 \\ & \hspace{0.17in} (\alpha_R)_1 - (\alpha_R)_2 - (\alpha_R)_3 = 0 \\ (\mbox{Cond. } 3): \; \; & \hspace{0.17in} (\alpha_R)_3 \geq (\alpha_D)_2 \geq (\alpha_R)_1 \geq 0. \end{aligned} \right.\]
We can choose $(1,0,1,0,1)$ so that this is balanced exterior forest.
\item
In order for the right exterior forest to be balanced, it is required that we find a nontrivial vector $\alpha = ((\alpha_R)_1, (\alpha_R)_2, (\alpha_R)_3, (\alpha_D)_1, (\alpha_D)_2) \in \mathbb{R}_{\geq 0}^5$, $\alpha \not= \mathbf{0}$, satisfying:
\[\left\{ \; \; \begin{aligned} (\mbox{Cond. } 1): \; \; & \hspace{0.17in} (\alpha_R)_1 = 0, \; (\alpha_D)_2 = 0 \\ (\mbox{Cond. } 2): \; \; & -(\alpha_R)_1 + (\alpha_R)_2 + (\alpha_R)_3 = 0 \\ & \hspace{0.17in} (\alpha_R)_1 - (\alpha_R)_2 - (\alpha_R)_3 = 0 \\ (\mbox{Cond. } 3): \; \; & \hspace{0.17in} (\alpha_R)_3 \geq (\alpha_D)_1 \geq (\alpha_R)_2 \geq 0. \end{aligned} \right.\]
Substituting Condition 1 into Condition 2 gives $(\alpha_R)_2 + (\alpha_R)_3 = 0$ which is inconsistent with the requirement from Condition 3 that $(\alpha_R)_3 \geq (\alpha_R)_2 \geq 0$ and at least one entry be nonzero. It follows that this is an unbalanced exterior forest.
\end{enumerate}
\end{example}

\subsection{Conditions for Extinction Events}
\label{mainsection}

We now present the main results of this paper, which are inspired by Theorem 1 and Corollary 2 of \cite{Brijder}. The proof of Theorem \ref{mainresult} is contained in Appendix \ref{appendixb}.

\begin{theorem}
\label{mainresult}
Consider a subconservative CRN and a $\mathcal{Y}$-admissible dom-CRN where $\mathcal{Y} \subseteq \mathcal{C}$ is an absorbing complex set on the dom-CRN.  Suppose that there is a complex $y \not\in \mathcal{Y}$ of the dom-CRN which is recurrent from a state $\mathbf{X} \in \mathbb{Z}_{ \geq 0}^m$ in the discrete state space CRN. Then every $\mathcal{Y}$-exterior forest of the dom-CRN is balanced.
\end{theorem}

\noindent This result places restrictions on the structure of a subconservative CRN that does not experience a guaranteed extinction event. We will be more frequently interested in when discrete extinction occurs, and therefore present the following corollary which follows immediately as the contrapositive of Theorem \ref{mainresult}.

\begin{corollary}
\label{maincorollary}
Consider a subconservative CRN and a $\mathcal{Y}$-admissible dom-CRN where $\mathcal{Y} \subseteq \mathcal{C}$ is an absorbing complex set on the dom-CRN. Suppose there is a $\mathcal{Y}$-exterior forest of the dom-CRN which is unbalanced. Then the discrete state space CRN has a guaranteed extinction event on $\mathcal{Y}^c$. 
\end{corollary}

\noindent Recall that an exterior forest is unbalanced if there is a set of equalities and inequalities on the edges of the dom-CRN which cannot be satisfied. The question of determining sufficient conditions for discrete extinction is therefore reduced to determining the feasibility of particular sets of equalities and inequalities.

Notice also that, even if a CRN permits many $\mathcal{Y}$-exterior forests, it is sufficient for a single one to be unbalanced for an extinction event to follow. Furthermore, the set of transient complexes corresponds to the set of complexes not in $\mathcal{Y}$. Note that this may contain terminal complexes in the original CRN (see Example \ref{example920}).

\begin{remark}
By convention, when applying Corollary \ref{maincorollary}, if no mention of an absorbing complex set $\mathcal{Y} \subseteq \mathcal{C}$ is made, it is assumed to be the set of terminal complexes in the dom-CRN.
\end{remark}

\begin{example}
Reconsider the CRN analyzed in Example \ref{example234}, \ref{example919}, and \ref{example920}. This CRN is conservative, and in Example \ref{example921} we showed that there is an admissible dom-CRN with an unbalanced exterior forest. It follows from Corollary \ref{maincorollary} that the discrete state space CRN has a guaranteed extinction event on the set of nonterminal complexes of the dom-CRN. That is, from all states $\mathbf{X} \in \mathbb{Z}_{\geq 0}^m$, there is guaranteed to be a time after which the count of the species is insufficient for any reaction from the complexes $X_1 + X_2$, $2X_2$, and $X_2$ to occur. This is consistent with our earlier observation that the state $\{ \mathbf{X}_1 = M, \mathbf{X}_2 = 0 \}$ where $M = \mathbf{X}_1(0) + \mathbf{X}_2(0)$ absorbs all trajectories through repeated application of the reactions $2X_2 \to X_1 + X_2$ and $X_2 \to X_1$. Notice that this pathway consists of the true reactions in the unbalanced exterior forest.
\end{example}

\subsection{EnvZ-OmpR Signaling Pathway}

In this section, we consider a CRN which was proposed as underlying the EnvZ/OmpR signaling pathway in \emph{Escherichia coli} in \cite{Sh-F}. 
This CRN has been studied previously with a discrete state space in the papers \cite{A-E-J,Brijder} where it was shown to exhibit a guarantee extinction event. We reconsider the CRN here to demonstrate the process of applying Corollary \ref{maincorollary} and also to demonstrate the advantages of our approach. In particular, the graphical method of constructing $\mathcal{Y}$-exterior forests suggests the pathways to extinction in the CRN.

\begin{example}
\label{example1}
Consider the following reaction mechanism, which was proposed by G. Shinar and M. Feinberg as underlying the EnvZ/OmpR signaling pathway in \emph{Escherichia coli} in the Supplemental Material of \cite{Sh-F}:
\begin{center}
\begin{tikzpicture}[auto, outer sep=3pt, node distance=2cm,>=latex']
\node (X1) {$X_1$};
\node [right of = X1, node distance = 2.5cm] (X2) {$X_2$};
\node [right of = X2, node distance = 2.5cm] (X3) {$X_3$};
\node [right of = X3, node distance = 2.5cm] (X4) {$X_4$};
\node [below of = X1, node distance = 1cm] (X4X5) {$X_4 + X_5$};
\node [right of = X4X5, node distance = 2.5cm] (X6) {$X_6$};
\node [right of = X6, node distance = 2.5cm] (X2X7) {$X_2 + X_7$};
\node [below of = X4X5, node distance = 1cm] (X3X7) {$X_3 + X_7$};
\node [right of = X3X7, node distance = 2.5cm] (X8) {$X_8$};
\node [right of = X8, node distance = 2.5cm] (X3X5) {$X_3 + X_5$};
\node [below of = X3X7, node distance = 1cm] (X1X7) {$X_1 + X_7$};
\node [right of = X1X7, node distance = 2.5cm] (X9) {$X_9$};
\node [right of = X9, node distance = 2.5cm] (X1X5) {$X_1 + X_5$};
\path[->,bend left=10] (X1) edge node [above left = -0.15cm] {\tiny $1$} (X2);
\path[->,bend left=10] (X2) edge node [below right = -0.15cm] {\tiny $2$} (X1);
\path[->,bend left=10] (X2) edge node [above left = -0.15cm] {\tiny $3$} (X3);
\path[->,bend left=10] (X3) edge node [below right = -0.15cm] {\tiny $4$} (X2);
\path[->] (X3) edge node [above = -0.15cm] {\tiny $5$} (X4);
\path[->,bend left=10] (X4X5) edge node [above left = -0.15cm] {\tiny $6$} (X6);
\path[->,bend left=10] (X6) edge node [below right = -0.15cm] {\tiny $7$} (X4X5);
\path[->] (X6) edge node [above = -0.15cm] {\tiny $8$} (X2X7);
\path[->,bend left=10] (X3X7) edge node [above left = -0.15cm] {\tiny $9$} (X8);
\path[->,bend left=10] (X8) edge node [below right = -0.15cm] {\tiny $10$} (X3X7);
\path[->] (X8) edge node [above = -0.15cm] {\tiny $11$} (X3X5);
\path[->,bend left=10] (X1X7) edge node [above left = -0.15cm] {\tiny $12$} (X9);
\path[->,bend left=10] (X9) edge node [below right = -0.15cm] {\tiny $13$} (X1X7);
\path[->] (X9) edge node [above = -0.15cm] {\tiny $14$} (X1X5);
\end{tikzpicture}
\end{center}
\noindent where $X_1 = \mbox{\emph{EnvZ-ADP}}$, $X_2 = \mbox{\emph{EnvZ}}$, $X_3 = \mbox{\emph{EnvZ-ATP}}$, $X_4 = \mbox{\emph{EnvZ}}_p$, $X_5 = \mbox{\emph{OmpR}}$, $X_6 = \mbox{\emph{EnvZ}}_p\mbox{\emph{-OmpR}}$, $X_7 = \mbox{\emph{OmpR}}_p$, $X_8 = \mbox{\emph{EnvZ-ATP-OmpR}}_p$, $X_9 = \mbox{\emph{EnvZ-ADP-OmpR}}_p$.

Consider the admissible dom-CRN with $\mathcal{D} = \{ X_1 + X_5 \stackrel{D_1}{\longrightarrow} X_1, X_1 + X_7 \stackrel{D_2}{\longrightarrow} X_1, X_2+X_7 \stackrel{D_3}{\longrightarrow} X_2, X_3 + X_5 \stackrel{D_4}{\longrightarrow} X_3, X_3 + X_7 \stackrel{D_5}{\longrightarrow} X_3\}$. The dom-CRN may be graphically represented as:

\begin{center}
\begin{tikzpicture}[auto, outer sep=3pt, node distance=2cm,>=latex']
\node (X1X7) {$X_1 + X_7$};
\node [below of = X1X7, node distance = 1.5cm] (X9) {$X_9$};
\node [below of = X9, node distance = 1.5cm] (X1X5) {$X_1 + X_5$};
\node [right of = X9, node distance = 2.5cm] (X1) {$X_1$};
\node [right of = X1, node distance = 2.5cm] (X2) {$X_2$};
\node [right of = X2, node distance = 2.5cm] (X3) {$X_3$};
\node [right of = X3, node distance = 2.5cm] (X4) {$X_4$};
\node [below of = X2, node distance = 1.5cm] (X3X7) {$X_3 + X_7$};
\node [right of = X3X7, node distance = 2.5cm] (X8) {$X_8$};
\node [right of = X8, node distance = 2.5cm] (X3X5) {$X_3 + X_5$};
\node [above of = X4, node distance = 1.5cm] (X4X5) {$X_4 + X_5$};
\node [left of = X4X5, node distance = 2.5cm] (X6) {$X_6$};
\node [left of = X6, node distance = 2.5cm] (X2X7) {$X_2 + X_7$};
\path[->,bend left=10] (X1) edge node [above left = -0.15cm] {\tiny $1$} (X2);
\path[->,bend left=10] (X2) edge node [below right = -0.15cm] {\tiny $2$} (X1);
\path[->,bend left=10] (X2) edge node [above left = -0.15cm] {\tiny $3$} (X3);
\path[->,bend left=10] (X3) edge node [below right = -0.15cm] {\tiny $4$} (X2);
\path[->] (X3) edge node [above = -0.15cm] {\tiny $5$} (X4);
\path[->,bend left=10] (X4X5) edge node [below right = -0.15cm] {\tiny $6$} (X6);
\path[->,bend left=10] (X6) edge node [above left = -0.15cm] {\tiny $7$} (X4X5);
\path[->] (X6) edge node [above = -0.15cm] {\tiny $8$} (X2X7);
\path[->,bend left=10] (X3X7) edge node [above left = -0.15cm] {\tiny $9$} (X8);
\path[->,bend left=10] (X8) edge node [below right = -0.15cm] {\tiny $10$} (X3X7);
\path[->] (X8) edge node [above = -0.15cm] {\tiny $11$} (X3X5);
\path[->,bend left=10] (X1X7) edge node [right = -0.1cm] {\tiny $12$} (X9);
\path[->,bend left=10] (X9) edge node [left = -0.1cm] {\tiny $13$} (X1X7);
\path[->] (X9) edge node [right = -0.1cm] {\tiny $14$} (X1X5);
\path[dashed,->] (X1X7) edge [above right = -0.15cm] node {\tiny $D_2$} (X1);
\path[dashed,->] (X1X5) edge [below right = -0.15cm] node {\tiny $D_1$} (X1);
\path[dashed,->] (X3X7) edge [left] node {\tiny $D_5$} (X3);
\path[dashed,->] (X3X5) edge [right] node {\tiny $D_4$} (X3);
\path[dashed,->] (X2X7) edge [right = -0.15cm] node {\tiny $D_3$} (X2);
\end{tikzpicture}
\end{center}

\noindent Consider furthermore the following exterior forest:
\begin{center}
\begin{tikzpicture}[auto, outer sep=3pt, node distance=2cm,>=latex']
\node (X1X7) {$X_1 + X_7$};
\node [below of = X1X7, node distance = 1.5cm] (X9) {$X_9$};
\node [below of = X9, node distance = 1.5cm] (X1X5) {$X_1 + X_5$};
\node [right of = X9, node distance = 2.5cm] (X1) {$X_1$};
\node [right of = X1, node distance = 2.5cm] (X2) {$X_2$};
\node [right of = X2, node distance = 2.5cm] (X3) {$X_3$};
\node [right of = X3, node distance = 2.5cm] (X4) {$X_4$};
\node [below of = X2, node distance = 1.5cm] (X3X7) {$X_3 + X_7$};
\node [right of = X3X7, node distance = 2.5cm] (X8) {$X_8$};
\node [right of = X8, node distance = 2.5cm] (X3X5) {$X_3 + X_5$};
\node [above of = X4, node distance = 1.5cm] (X4X5) {$X_4 + X_5$};
\node [left of = X4X5, node distance = 2.5cm] (X6) {$X_6$};
\node [left of = X6, node distance = 2.5cm] (X2X7) {$X_2 + X_7$};
\path[->,line width=0.75mm,bend left=10,red] (X1) edge node [above left = -0.15cm] {\tiny $\mathbf{1}$} (X2);
\path[->,bend left=10] (X2) edge node [below right = -0.15cm] {\tiny $2$} (X1);
\path[->,line width=0.75mm,bend left=10,red] (X2) edge node [above left = -0.15cm] {\tiny $\mathbf{3}$} (X3);
\path[->,bend left=10] (X3) edge node [below right = -0.15cm] {\tiny $4$} (X2);
\path[->,line width=0.75mm,red] (X3) edge node [above = -0.15cm] {\tiny $\mathbf{5}$} (X4);
\path[->,line width=0.75mm,bend left=10,red] (X4X5) edge node [below right = -0.15cm] {\tiny $\mathbf{6}$} (X6);
\path[->,bend left=10] (X6) edge node [above left = -0.15cm] {\tiny $7$} (X4X5);
\path[->,line width=0.75mm,red] (X6) edge node [above = -0.15cm] {\tiny $\mathbf{8}$} (X2X7);
\path[->,bend left=10] (X3X7) edge node [above left = -0.15cm] {\tiny $9$} (X8);
\path[->,line width=0.75mm,bend left=10,red] (X8) edge node [below right = -0.15cm] {\tiny $\mathbf{10}$} (X3X7);
\path[->] (X8) edge node [above = -0.15cm] {\tiny $11$} (X3X5);
\path[->,bend left=10] (X1X7) edge node [right = -0.1cm] {\tiny $12$} (X9);
\path[->,line width=0.75mm,bend left=10,red] (X9) edge node [left = -0.1cm] {\tiny $\mathbf{13}$} (X1X7);
\path[->] (X9) edge node [right = -0.1cm] {\tiny $14$} (X1X5);
\path[dashed,line width=0.75mm,->,red] (X1X7) edge [above right] node {\tiny $\mathbf{D_2}$} (X1);
\path[dashed,line width=0.75mm,->,red] (X1X5) edge [below right] node {\tiny $\mathbf{D_1}$} (X1);
\path[dashed,line width=0.75mm,->,red] (X3X7) edge [left] node {\tiny $\mathbf{D_5}$} (X3);
\path[dashed,line width=0.75mm,->,red] (X3X5) edge [right] node {\tiny $\mathbf{D_4}$} (X3);
\path[dashed,line width=0.75mm,->,red] (X2X7) edge [right = -0.15cm] node {\tiny $\mathbf{D_3}$} (X2);
\end{tikzpicture}
\end{center}

\noindent In the highlighted structure (bold red), there is a unique path from every complex to the terminal complex $X_4$. It can be seen directly that this exterior forest is unbalanced by noting that we need a vector $\alpha = (\alpha_R,\alpha_D) \in \mathbb{Z}_{\geq 0}^{19}$, $\alpha \not= \mathbf{0}$, which has support on a subset of the red highlighted structure above. To satisfy Condition 2 of Definition \ref{balancing}, we need to satisfy $\alpha_R \in \ker(\Gamma)$. We can check that $\ker(\Gamma) \cap \mathbb{R}_{\geq 0}^r$ has the generators:
	\[\begin{array}{rllllllllllllll} \mbox{\emph{reaction:}} & \; 1 & 2 & 3 & 4 & 5 & 6 & 7 & 8 & 9 & 10 & 11 & 12 & 13 & 14 \\ \hline \\[-0.2cm]
\ker(\Gamma) \cap \mathbb{R}_{\geq 0}^r = \{ & (1, & 1, & 0, & 0, & 0, & 0, & 0, & 0, & 0, & \;0, & \; 0, & \; 0, & \;0, & \; 0),\\
& (0, & 0, & 1, & 1, & 0, & 0, & 0, & 0, & 0, & \;0, & \; 0, & \; 0, & \;0, & \; 0),\\
& (0, & 0, & 0, & 0, & 0, & 1, & 1, & 0, & 0, & \;0, & \; 0, & \; 0, & \;0, & \; 0),\\
& (0, & 0, & 0, & 0, & 0, & 0, & 0, & 0, & 1, & \;1, & \; 0, & \; 0, & \;0, & \; 0),\\
& (0, & 0, & 0, & 0, & 0, & 0, & 0, & 0, & 0, & \;0, & \; 0, & \; 1, & \;1, & \; 0),\\
& (0, & 0, & 1, & 0, & 1, & 1, & 0, & 1, & 1, & \;0, & \; 1, & \; 0, & \;0, & \; 0),\\
& (0, & 0, & 1, & 0, & 1, & 1, & 0, & 1, & 0, & \;0, & \; 0, & \; 1, & \;0, & \; 1) \; \; \;\}
\end{array}\]
\noindent The first five vectors correspond to reversible reaction pairs in the CRN and so may be ignored. In order to obtain a nontrivial vector $\alpha_R$, we require $(\alpha_R)_5 > 0$. To build such a vector using the sixth vector yields a vector with support on $(\alpha_R)_{11}$ while building it out of the seventh vector yields a vector with support on $(\alpha_R)_{14}$. Neither of these options is consistent with Condition 1 of Definition \ref{balancing} so that the exterior forest is unbalanced. It follows by Corollary \ref{maincorollary} that the discrete state space CRN has a guaranteed extinction event, and that every complex except $X_4$ is transient. In fact, all trajectories are absorbed by a state where $\mathbf{X}_4 > 0$, $\mathbf{X}_7 > 0$, and $\mathbf{X}_i = 0$ for $i \in \{ 1, 2, 3, 5, 6, 8, 9\}$.

This result was previously obtained in \cite{A-E-J} and also proved for a simplified CRN in \cite{Brijder}. The construction of the dom-CRN, and computational implementation, is unique. This method also suggests a pathway toward extinction through the reactions in the unbalanced exterior forest. Such a pathway was not apparent by the methods of either \cite{A-E-J} or \cite{Brijder}.

\end{example}

\subsection{Further Examples}

In this section, we provide further examples which demonstrate how to apply Corollary \ref{maincorollary}, and also demonstrate the necessity of several of the technical conditions required of the result. Example \ref{example000} presents a CRN which can be shown to have an extinction event for an absorbing complex set $\mathcal{Y} \subseteq \mathcal{C}$ which is not the set of terminal complexes in the dom-CRN. Example \ref{example001} presents a CRN which does not have a guaranteed extinction event, but which can be shown to have an unbalanced exterior forest if we do not insist on the underlying dom-CRN being admissible. Example \ref{example999} demonstrates that including Condition 3 of Definition \ref{balancing} allows further classification of CRNs with extinction events than would be possible otherwise. Examples \ref{example100} and \ref{example101} provide CRNs which show that the conditions of Corollary \ref{maincorollary} are sufficient, but not necessary, for a guaranteed extinction event to occur.

\begin{example}
\label{example000}
It is natural to wonder whether, when applying Corollary \ref{maincorollary}, there is an advantage to generalizing the set of terminal complexes to an absorbing complex set $\mathcal{Y} \subseteq \mathcal{C}$. To show that there is, consider the following CRN:
\begin{center}
\begin{tikzpicture}[auto, outer sep=3pt, node distance=2cm,>=latex']
\node (C1) {$2X_1$};
\node [right of = C1, node distance = 2.5cm] (C2) {$X_2+X_3$};
\node [right of = C2, node distance = 2.5cm] (C3) {$2X_3$};
\node [right of = C3, node distance = 2.5cm] (C4) {$2X_2$};
\path[->] (C1) edge node [above = -0.15cm] {\tiny $1$} (C2);
\path[->] (C2) edge node [above = -0.15cm] {\tiny $2$} (C3);
\path[->,bend left = 10] (C3) edge node [above left = -0.15cm] {\tiny $3$} (C4);
\path[->,bend left = 10] (C4) edge node [below right = -0.15cm] {\tiny $4$} (C3);
\end{tikzpicture}
\end{center}
There are no domination relations so that the only dom-CRN corresponds to the CRN shown, and it is trivially admissible. The only exterior forest consists of all reactions. Notable, it contains reactions 1 and 2 on the nonterminal component. We can easily determine that $\alpha = (\alpha_1,\alpha_2,\alpha_3,\alpha_4) = (0,2,1,0)$ satisfies the conditions of Definition \ref{balancing} and therefore that this exterior forest is balanced. Therefore, Corollary \ref{maincorollary} does not apply and we may not conclude that an extinction event occurs.

Consider instead taking $\mathcal{Y} = \{ X_2 + X_3, 2X_3, 2X_2 \}$. This set is absorbing and contains every terminal complex of the CRN. The only exterior forest again contains all reactions but only reaction 1 is $\mathcal{Y}$-exterior. Since there is no balancing vector $\alpha$ for which $\alpha_1 \not=0$, we may conclude by Corollary \ref{maincorollary} that there is a guaranteed extinction event on $\mathcal{Y}^c = \{ 2X_1 \}$. In fact, we can see this directly since repeated application of reaction 1 will deplete $X_1$ and there are no pathways by which to replenish it.

\end{example}

\begin{example}
\label{example001}
It is natural to wonder whether it is necessary to insist on dom-CRNs being admissible. To show that removing this assumption from Corollary \ref{maincorollary} can lead to misclassification, consider the following CRN:
\begin{center}
\begin{tikzpicture}[auto, outer sep=3pt, node distance=2cm,>=latex']
\node  (C1) {$2X_1$};
\node [right of = C1, node distance = 2.5cm] (C2) {$2X_2$};
\node [right of = C2, node distance = 2.5cm] (C3) {$X_2$};
\node [right of = C3, node distance = 2.5cm] (C4) {$X_3$};
\path[->,bend left=10] (C1) edge node [above left = -0.15cm] {\tiny $1$} (C2);
\path[->,bend left=10] (C2) edge node [below right = -0.15cm] {\tiny $2$} (C1);
\path[->,bend left=10] (C3) edge node [above left = -0.15cm] {\tiny $3$} (C4);
\path[->,bend left=10] (C4) edge node [below right = -0.15cm] {\tiny $4$} (C3);
\end{tikzpicture}
\end{center}
The CRN has only the single domination relation $X_2 \leq 2X_2$. Since the corresponding domination relation $2X_2 \to X_2$ leads to a terminal component in any resulting dom-CRN, we may not add it, so that the only admissible dom-CRN corresponds to the original CRN.

Suppose, however, that we do not insist on dom-CRNs being admissible. Specifically, suppose we allow the following dom-CRN:
\begin{center}
\begin{tikzpicture}[auto, outer sep=3pt, node distance=2cm,>=latex']
\node  (C1) {$2X_1$};
\node [right of = C1, node distance = 2.5cm] (C2) {$2X_2$};
\node [right of = C2, node distance = 2.5cm] (C3) {$X_2$};
\node [right of = C3, node distance = 2.5cm] (C4) {$X_3$};
\path[->,bend left=10] (C1) edge node [above left = -0.15cm] {\tiny $1$} (C2);
\path[->,bend left=10] (C2) edge node [below right = -0.15cm] {\tiny $2$} (C1);
\path[->,bend left=10] (C3) edge node [above left = -0.15cm] {\tiny $3$} (C4);
\path[->,bend left=10] (C4) edge node [below right = -0.15cm] {\tiny $4$} (C3);
\path[dashed,->] (C2) edge node [above = -0.15cm] {\tiny $D$} (C3);
\end{tikzpicture}
\end{center}
The only exterior forest is given in bold red as follows:
\begin{center}
\begin{tikzpicture}[auto, outer sep=3pt, node distance=2cm,>=latex']
\node  (C1) {$2X_1$};
\node [right of = C1, node distance = 2.5cm] (C2) {$2X_2$};
\node [right of = C2, node distance = 2.5cm] (C3) {$X_2$};
\node [right of = C3, node distance = 2.5cm] (C4) {$X_3$};
\path[red,line width=0.75mm,->,bend left=10] (C1) edge node [above left = -0.15cm] {\tiny $\mathbf{1}$} (C2);
\path[->,bend left=10] (C2) edge node [below right = -0.15cm] {\tiny $2$} (C1);
\path[red,line width=0.75mm,->,bend left=10] (C3) edge node [above left = -0.15cm] {\tiny $\mathbf{3}$} (C4);
\path[red,line width=0.75mm,->,bend left=10] (C4) edge node [below right = -0.15cm] {\tiny $\mathbf{4}$} (C3);
\path[red,line width=0.75mm,dashed,->] (C2) edge node [above = -0.15cm] {\tiny $\mathbf{D}$} (C3);
\end{tikzpicture}
\end{center}
Notice that we have included the terminal reactions in the exterior forest. In order to be balanced, we must find a vector $\alpha = (\alpha_1, \alpha_2, \alpha_3, \alpha_4, \alpha_D)$ which is nonzero on at least one of the nonterminal reactions $\alpha_1$ and $\alpha_2$, such that
\[\left\{ \; \; \begin{aligned} (\mbox{Cond.} \; 1): \; \; & \hspace{0.17in}  \alpha_2 = 0 \\ (\mbox{Cond.} \; 2): \; \; &  -2\alpha_1 + 2\alpha_2 = 0 \\ & \hspace{0.17in} 2\alpha_1 - 2\alpha_2 - \alpha_3 + \alpha_4 = 0 \\ & \hspace{0.17in} \alpha_3 - \alpha_4 = 0 \\ (\mbox{Cond.}\; 3): \; \; & \hspace{0.17in} \alpha_D \geq \alpha_1 \geq 0. \end{aligned} \right.\]
Conditions 1 and 2 imply that $\alpha_1 = 0$ so that $\alpha$ is does not have support on the nonterminal portion of the dom-CRN. It follows that the exterior forest is unbalanced. Note, however, that Corollary \ref{maincorollary} remains silent since the presented dom-CRN is not admissible. Since all reactions of the discrete state space CRN are recurrent whenever $X_1(0) + X_2(0) + X_3(0) \geq 3$, this example highlights the importance of the assumption that dom-CRNs be admissible.

\end{example}


\begin{example}
\label{example999}
It is natural to wonder whether Condition 3 of Definition \ref{balancing} is useful in classifying discrete state space CRNs with extinction events. To see that it can be, consider the following CRN:
\begin{center}
\begin{tikzpicture}[auto, outer sep=3pt, node distance=2cm,>=latex']
\node  (C4) {$X_1+X_2$};
\node [right of = C4, node distance = 2.5cm] (C5) {$2X_1$};
\node [right of = C5, node distance = 2.5cm] (C6) {$2X_2$};
\path[->,bend left=10] (C4) edge node [above left = -0.15cm] {\tiny $1$} (C5);
\path[->,bend left=10] (C5) edge node [below right = -0.15cm] {\tiny $2$} (C4);
\path[->] (C5) edge node [above = -0.15cm] {\tiny $3$} (C6);
\end{tikzpicture}
\end{center}
There are no domination relations so the dom-CRN coincides with the original CRN. We have only the following exterior forest in bold red:
\begin{center}
\begin{tikzpicture}[auto, outer sep=3pt, node distance=2cm,>=latex']
\node (C1) {$X_1+X_2$};
\node [right of = C1, node distance = 2.5cm] (C2) {$2X_1$};
\node [right of = C2, node distance = 2.5cm] (C3) {$2X_2.$};
\path[->,line width=0.75mm,red,bend left=10] (C1) edge node [above left = -0.15cm] {\tiny $\mathbf{1}$} (C2);
\path[->,bend left=10] (C2) edge node [below right = -0.15cm] {\tiny $2$} (C1);
\path[->,line width=0.75mm,red] (C2) edge node [above = -0.15cm] {\tiny $\mathbf{3}$} (C3);
\end{tikzpicture}
\end{center}
In order for this exterior forest to be balanced, we need to have a vector $\alpha = (\alpha_1, \alpha_2, \alpha_3)$, $\alpha \not= \mathbf{0}$, which satisfies the following equalities and inequalities:
\[\left\{ \; \; \begin{aligned} (\mbox{Cond.} \; 1): \; \; & \hspace{0.17in} \alpha_2 = 0 \\ (\mbox{Cond.} \; 2): \; \; &  \hspace{0.17in}\alpha_1 - \alpha_2  -2 \alpha_3 = 0 \\ & -\alpha_1 + \alpha_2 + 2 \alpha_3 = 0 \\ (\mbox{Cond.} \; 3): \; \; & \hspace{0.17in} \alpha_3 \geq \alpha_1 \geq 0. \end{aligned} \right.\]
Condition 1 reduces Condition 2 to $\alpha_1 = 2 \alpha_3$, so that, combining with Condition 3, we have
\[\alpha_3 \geq \alpha_1 = 2\alpha_3 \geq 0.\]
This can only be satisfied by $\alpha_1 = 0$ and $\alpha_3 = 0$, which is a violation. It follows that the exterior forest is unbalanced and therefore, by Corollary \ref{maincorollary}, the discrete state space CRN has a guaranteed extinction event on the nonterminal complexes $X_1 + X_2$ and $2X_1$. Note, however, that the vector $\alpha = (2,0,1)$ satisfies conditions 1 and 2 of Definition \ref{balancing}. It follows that Condition 3 of Definition \ref{balancing} allows furthermore classification of CRNs with extinction events than conditions 1 and 2 allow by themselves. Note also that this CRN is also not classified as having a guaranteed extinction event by Corollary 2 of \cite{Brijder}.

\end{example}


\begin{example}
\label{example100}
It is natural to wonder whether the conditions of Theorem \ref{mainresult} and Corollary \ref{maincorollary} are necessary as well as sufficient for a discrete state space CRN to have an extinction event. To show that they are sufficient only, consider the following CRN, which is simplified from the CRN in Eqn.\ (66) of \cite{N-G-N2011} and reproduced as Eqn.\ (49) in the Supplemental Material for \cite{A-E-J}:
\begin{center}
\begin{tikzpicture}[auto, outer sep=3pt, node distance=2cm,>=latex']
\node (C1) {\; \; \; \; $X_1$};
\node [right of = C1, node distance = 3cm] (C2) {$X_2$ \; \; \; \;};
\node [below of = C1, node distance = 0.75cm] (C3) {$X_2 + X_3$};
\node [right of = C3, node distance = 3cm] (C4) {$X_1+X_3$};
\node [below of = C3, node distance = 0.75cm] (C5) {$X_3 + X_4$};
\node [right of = C5, node distance = 3 cm] (C6) {$X_1 + X_4$};
\path[->] (C1) edge node [above = -0.15cm] {\tiny $1$} (C2);
\path[->] (C3) edge node [above = -0.15cm] {\tiny $2$} (C4);
\path[->, bend left = 10] (C5) edge node [above = -0.15cm] {\tiny $3$} (C6);
\path[->, bend left = 10] (C6) edge node [below = -0.15cm] {\tiny $4$} (C5);
\end{tikzpicture}
\end{center}
The CRN has a guaranteed discrete extinction event, since $X_3$ may convert into $X_1$ through reaction 3, then $X_1$ may convert into $X_2$ through reaction 1. This shuts down all reactions.

To show that Corollary \ref{maincorollary} is incapable of affirming this extinction event, it is necessary to show that every $\mathcal{Y}$-exterior forest of every $\mathcal{Y}$-admissible dom-CRN is unbalanced. We start by considering the terminal complexes and the set $\mathcal{D} = \{ X_1 + X_3 \stackrel{D_1}{\longrightarrow} X_1, X_1 + X_4 \stackrel{D_2}{\longrightarrow} X_1 \}$. This gives the following dom-CRN:
\begin{center}
\begin{tikzpicture}[auto, outer sep=3pt, node distance=2cm,>=latex']
\node (C3) {$X_2 + X_3$};
\node [right of = C3, node distance = 3cm] (C4) {$X_1+X_3$};
\node [below of = C3, node distance = 1cm] (C5) {$X_3 + X_4$};
\node [right of = C5, node distance = 3cm] (C6) {$X_1 + X_4$};
\node [below of = C4, node distance = 0.5cm] (ghost) {};
\node [right of = ghost, node distance = 2.5cm] (C1) {$X_1$};
\node [right of = C1, node distance = 2.5cm] (C2) {$X_2$};
\path[->] (C1) edge node [above = -0.15cm] {\tiny $1$} (C2);
\path[->] (C3) edge node [above = -0.15cm] {\tiny $2$} (C4);
\path[->, bend left = 10] (C5) edge node [above = -0.15cm] {\tiny $3$} (C6);
\path[->, bend left = 10] (C6) edge node [below = -0.15cm] {\tiny $4$} (C5);
\path[dashed,->] (C4) edge node [above = -0.15cm] {\tiny $D_1$} (C1);
\path[dashed,->] (C6) edge node [below = -0.15cm] {\tiny $D_2$} (C1);
\end{tikzpicture}
\end{center}
This dom-CRN is admissible and admits only a single exterior forest in bold red:
\begin{center}
\begin{tikzpicture}[auto, outer sep=3pt, node distance=2cm,>=latex']
\node (C3) {$X_2 + X_3$};
\node [right of = C3, node distance = 3cm] (C4) {$X_1+X_3$};
\node [below of = C3, node distance = 1cm] (C5) {$X_3 + X_4$};
\node [right of = C5, node distance = 3cm] (C6) {$X_1 + X_4$};
\node [below of = C4, node distance = 0.5cm] (ghost) {};
\node [right of = ghost, node distance = 2.5cm] (C1) {$X_1$};
\node [right of = C1, node distance = 2.5cm] (C2) {$X_2$};
\path[red,line width=0.75mm,->] (C1) edge node [above = -0.15cm] {\tiny $\mathbf{1}$} (C2);
\path[red,line width=0.75mm,->] (C3) edge node [above = -0.15cm] {\tiny $\mathbf{2}$} (C4);
\path[red,line width=0.75mm,->, bend left = 10] (C5) edge node [above = -0.15cm] {\tiny $\mathbf{3}$} (C6);
\path[->, bend left = 10] (C6) edge node [below = -0.15cm] {\tiny $4$} (C5);
\path[red,line width=0.75mm,dashed,->] (C4) edge node [above = -0.15cm] {\tiny $\mathbf{D_1}$} (C1);
\path[red,line width=0.75mm,dashed,->] (C6) edge node [below = -0.15cm] {\tiny $\mathbf{D_2}$} (C1);
\end{tikzpicture}
\end{center}
This forest is balanced if we have a nontrivial vector $\alpha = ((\alpha_R)_1, (\alpha_R)_2, (\alpha_R)_3, (\alpha_R)_4, (\alpha_D)_1,$\\$ (\alpha_D)_2)$, $\alpha_R \not= \mathbf{0}$, which satisfies the following:
\begin{equation}
\label{conditions}
\left\{ \; \; \begin{aligned} (\mbox{Cond.} \; 1): \; \; & \hspace{0.17in} (\alpha_R)_4 = 0 \\ (\mbox{Cond.} \; 2): \; \; & -(\alpha_R)_1 +(\alpha_R)_2 + (\alpha_R)_3 -(\alpha_R)_4 = 0 \\ & \hspace{0.17in} (\alpha_R)_1 - (\alpha_R)_2 = 0\\ & -(\alpha_R)_3 + (\alpha_R)_4 = 0 \\ (\mbox{Cond.} \; 3): \; \; & \hspace{0.17in} (\alpha_D)_1 + (\alpha_D)_2 \geq (\alpha_R)_1 \geq 0 \\ & \hspace{0.17in} (\alpha_D)_1 \geq (\alpha_R)_2 \\ & \hspace{0.17in} (\alpha_D)_2 \geq (\alpha_R)_3. \end{aligned} \right.
\end{equation}
This can be satisfied by the vector $\alpha = (1,1,0,0,1,0)$. It follows that the forest is balanced, and since this is the only exterior forest for the given dom-CRN, no conclusion may be reached as a result of Corollary \ref{maincorollary}.

We now consider more general absorbing complex sets $\mathcal{Y} \subseteq \mathcal{C}$. Notice that any potential $\mathcal{Y}$ which contains a subset of $\{ X_2, X_1, X_2+X_3, X_1 + X_3 \}$ can be balanced by the $\alpha$ above, with perhaps different support on $\alpha_D$. If $X_1 \in \mathcal{Y}$, however, we must have $\mathcal{D} = \emptyset$ in order for the dom-CRN to be $\mathcal{Y}$-admissible. Otherwise, we would have an edge in $\mathcal{D}$ which would lead to $\mathcal{Y}$. For $\mathcal{D} = \emptyset$, however, we have that $ X_3+X_4$ and $X_1+X_4$ are terminal in the dom-CRN and therefore $X_3 + X_4$ and $X_1 + X_4$ must be included in $\mathcal{Y}$. This leaves $\mathcal{Y} = \mathcal{C}$ which has an empty exterior forest. There are no other cases to consider, so we are done.

It follows that every $\mathcal{Y}$-exterior forest of every $\mathcal{Y}$-admissible dom-CRN is balanced. Since the CRN has a guaranteed extinction event, however, it follows that the conditions of Corollary \ref{maincorollary} are not necessary for extinction events in discrete state space CRNs.
\end{example}

\begin{example}
\label{example101}

To show that the gap raised in Example \ref{example100} may not be easily overcome by structural considerations alone, consider the following CRN:
\begin{center}
\begin{tikzpicture}[auto, outer sep=3pt, node distance=2cm,>=latex']
\node (C1) {\; \; \; \; $X_1$};
\node [right of = C1, node distance = 3cm] (C2) {$X_2$ \; \; \; \;};
\node [below of = C1, node distance = 0.75cm] (C3) {$X_2 + X_4$};
\node [right of = C3, node distance = 3cm] (C4) {$X_1+X_4$};
\node [below of = C3, node distance = 0.75cm] (C5) {$X_3 + X_5$};
\node [right of = C5, node distance = 3 cm] (C6) {$X_1 + X_5$};
\path[->] (C1) edge node [above = -0.15cm] {\tiny $1$} (C2);
\path[->] (C3) edge node [above = -0.15cm] {\tiny $2$} (C4);
\path[->, bend left = 10] (C5) edge node [above = -0.15cm] {\tiny $3$} (C6);
\path[->, bend left = 10] (C6) edge node [below = -0.15cm] {\tiny $4$} (C5);
\end{tikzpicture}
\end{center}
This is the CRN in Example \ref{example100} with $X_3$ replaced with $X_4$ in the reaction 2, and $X_4$ replaced with $X_5$ in reactions 3 and 4. Examples \ref{example100} and \ref{example101} share significant structural data, including connectivity of paths, domination relations between complexes, and ker$(\Gamma)$.

Taking $\mathcal{D} = \{ X_1 + X_4 \to X_1, X_1 + X_5 \to X_1 \}$ gives the following admissible dom-CRN:
\begin{center}
\begin{tikzpicture}[auto, outer sep=3pt, node distance=2cm,>=latex']
\node (C3) {$X_2 + X_4$};
\node [right of = C3, node distance = 3cm] (C4) {$X_1+X_4$};
\node [below of = C3, node distance = 1cm] (C5) {$X_3 + X_5$};
\node [right of = C5, node distance = 3cm] (C6) {$X_1 + X_5$};
\node [below of = C4, node distance = 0.5cm] (ghost) {};
\node [right of = ghost, node distance = 2.5cm] (C1) {$X_1$};
\node [right of = C1, node distance = 2.5cm] (C2) {$X_2$};
\path[->] (C1) edge node [above = -0.15cm] {\tiny $1$} (C2);
\path[->] (C3) edge node [above = -0.15cm] {\tiny $2$} (C4);
\path[->, bend left = 10] (C5) edge node [above = -0.15cm] {\tiny $3$} (C6);
\path[->, bend left = 10] (C6) edge node [below = -0.15cm] {\tiny $4$} (C5);
\path[dashed,->] (C4) edge node [above = -0.15cm] {\tiny $D_1$} (C1);
\path[dashed,->] (C6) edge node [below = -0.15cm] {\tiny $D_2$} (C1);
\end{tikzpicture}
\end{center}
We arrive at the same balancing equalities and inequalities (\ref{conditions}) as Example \ref{example100}, so that every $\mathcal{Y}$-exterior forest on this dom-CRN is balanced. Since the connectivity and domination relations are shared with Example \ref{example100}, we can exhaust nontrivial $\mathcal{Y}$-admissible dom-CRNs in the same way, and we conclude that Corollary \ref{maincorollary} is inconclusive.

In contrast to Example \ref{example100}, this example does not exhibit an extinction event for most initial conditions. Provided $\mathbf{X}_4 > 0$, $\mathbf{X}_5 > 0$, and any one of $\mathbf{X}_1$, $\mathbf{X}_2$, and $\mathbf{X}_3$ is positive, every complex is recurrent. This analysis suggests that comprehensive conditions for extinction events must depend on further structural information than that considered in this paper.

\end{example}

\section{Conclusions and Future Work}
\label{conclusions}

In this paper, we have presented novel conditions (Theorem \ref{mainresult} and Corollary \ref{maincorollary}) on the structure of a CRN that are sufficient to guarantee that the corresponding CRN 
 exhibits an extinction event. The conditions presented generalize the dependence on terminal SLCs in \cite{A-E-J} and \cite{Brijder}, and also produces a system of equalities and inequalities which can be directly verified. Our conditions have the additional advantage of being fundamentally graphical in nature and suggesting pathways to extinction.

This work raises several promising avenues for future work:
\begin{enumerate}
\item
While Corollary \ref{maincorollary} gives sufficient conditions for discrete extinction, they are not necessary (see Examples \ref{example100} and \ref{example101}). This raises the question of whether there are structural conditions which are both sufficient and necessary for discrete extinction and, if so, which further structural components of the CRN might be utilized in such a result.
\item
The conditions of Corollary \ref{maincorollary} consist of a system of equalities and inequalities. This suggests a computational implementation amenable, in particular, to the methods of linear programming. Linear programming has already been used widely in CRNT for verifying CRNs with desirable structural properties \cite{J4,J-S4,J-P-D,Sz2}. This will be explored and utilized to characterize CRNs with extinction events in the companion paper \cite{J2017}.
\end{enumerate}

\section*{Acknowledgments}

MDJ and DFA were supported by  Army Research Office grant W911NF-14-1-0401.  DFA was also supported by NSF-DMS-1318832 and MDJ was also supported by the Henry Woodward Fund.  GC was supported by NSF-DMS-1412643. RB is a postdoctoral fellow of the Research Foundation -- Flanders (FWO).

\newpage

\begin{appendices}
\section{Proof of Lemma \ref{terminallemma}}
\label{appendixa}

\noindent \textbf{Lemma \ref{terminallemma}.} \emph{If a CRN is subconservative, then for any dom-CRN: (i) the SLCs of the CRN and the dom-CRN coincide, and (ii) every terminal SLC of the dom-CRN is a terminal SLC of the CRN.}

\begin{proof} \emph{Proof of (i):} Consider a subconservative CRN and dom-CRN. Since the reactions of the CRN are contained in the reactions of the dom-CRN, it follows that the SLCs of CRN remain strongly connected in the dom-CRN and therefore are contained in the SLCs of the dom-CRN.

Now suppose that there is an SLC of the dom-CRN which is not contained in any SLC of the CRN. It follows that there are SLCs $W, W' \in \mathcal{W}$ of the CRN such that is a path in the dom-CRN from some complex $y_0 \in W$ to some complex $y'_0 \in W'$, and there is a path in the dom-CRN from some complex $y'_1 \in W'$ to some complex $y_1 \in W$. Since $W$ and $W'$ are strongly connected, we can create a cycle in the dom-CRN by constructing a path from $y_0$ to $y_0'$ to $y_1'$ to $y_1$ back to $y_0$. Furthermore, since this is not a cycle in the CRN (otherwise, $W$ and $W'$ would not be maximally strongly connected in the CRN), we have that there is at least one reaction in this cycle which is from $\mathcal{D}$.

We now index the complexes in the cycle so that, if there are $d' \geq 1$ reactions from $\mathcal{D}$ in the cycle, we have the following segments in between these reactions:
\begin{equation}
\label{cycle}
\begin{split}
& y_1^{(1)} \to y_2^{(1)} \to \cdots \to y_{n_1}^{(1)} \\
& y_1^{(2)} \to y_2^{(2)} \to \cdots \to y_{n_2}^{(2)} \\
& \hspace{0.8in} \vdots \\
& y_1^{(d')} \to y_2^{(d')} \to \cdots \to y_{n_{d'}}^{(d')}.
\end{split}
\end{equation}
We take the segments above to be connected by reactions in $\mathcal{R}$. We also take $y_1^{(i+1)} \leq y_{n_i}^{(i)}$ and $y_1^{(1)} \leq y_{n_{d'}}^{(d')}$ so that the endpoints of successive segments are joined together by $\left(y_{n_i}^{(i)}, y_1^{(i+1)} \right) \in \mathcal{D}$.

Let $\alpha \in \mathbb{Z}_{\geq 0}^r$ denote the vector of counts of the reactions in (\ref{cycle}), and take $y_1^{(1)} = y_1^{(d'+1)}$. It follows that
\begin{equation}
\label{10}
\begin{split} \Gamma \alpha & = \sum_{i=1}^{d'} \sum_{j=1}^{n_i-1} \left( y_{j+1}^{(i)} - y_{j}^{(i)} \right) \\
& = \sum_{i=1}^{d'} \left( y_{n_i}^{(i)} - y_1^{(i)} \right) \\ & = \sum_{i=1}^{d'} \left( y_{n_i}^{(i)} - y_1^{(i+1)} \right) \geq \mathbf{0} \end{split}
\end{equation}
by the domination relations $y_1^{(i+1)} \leq y_{n_i}^{(i)}$. Since the CRN is subconservative, it follows that there is a $\mathbf{c} \in \mathbb{R}_{> 0}^m$ such that $\mathbf{c}^T \Gamma \leq \mathbf{0}$. It follows that we have
\[0 \leq [\mathbf{c}^T \Gamma] \alpha = \mathbf{c}^T [ \Gamma \alpha ] > 0\]
where the last strict inequality follows from $c_i > 0$ for $i \in \{1, \ldots, m\}$ and the observation that at least one component in (\ref{10}) must be strictly greater than zero since the complexes of the CRN are stoichiometrically distinct. This is a contradiction. It follows that such a cycle does not exist in the dom-CRN so that $W$ and $W'$ are SLCs of the CRN. The SLCs of the CRN and dom-CRN therefore coincide and (i) is shown.\\

\noindent \emph{Proof of (ii):} Note that (i) guarantees that the CRN and dom-CRN share the same set of SLCs which we will denote $\mathcal{W}$. Suppose that $W \in \mathcal{W}$ is terminal in the dom-CRN but not in the CRN. This implies that there is a reaction $(y,y') \in \mathcal{R}$ where $y \in W$ and $y' \not\in W$; however, this reaction is included in the dom-CRN so that $W$ may not be terminal in the dom-CRN. It follows that every terminal SLC of the dom-CRN is a terminal SLC of the CRN, and (ii) is shown.
\end{proof}

\section{Proof of Theorem \ref{mainresult}}
\label{appendixb}

\noindent \textbf{Theorem \ref{mainresult}.} \emph{Consider a subconservative CRN and a $\mathcal{Y}$-admissible dom-CRN where $\mathcal{Y} \subseteq \mathcal{C}$ is an absorbing complex set on the dom-CRN.  Suppose that there is a complex $y \not\in \mathcal{Y}$ of the dom-CRN which is recurrent from a state $\mathbf{X} \in \mathbb{Z}_{ \geq 0}^m$ in the discrete state space CRN. Then every $\mathcal{Y}$-exterior forest of the dom-CRN is balanced.}


\begin{remark} The following proof is  inspired by the proof of Theorem 1 in \cite{Brijder}. The notation has been adapted to that of CRNT.
\end{remark}


\begin{proof}

Consider a subconservative CRN and a $\mathcal{Y}$-admissible dom-CRN where $\mathcal{Y} \subseteq \mathcal{C}$ is an absorbing complex set on the dom-CRN.  Suppose that there is a complex $y \not\in \mathcal{Y}$ of the dom-CRN which is recurrent from a state $\mathbf{X} \in \mathbb{Z}_{ \geq 0}^m$. We will show that every $\mathcal{Y}$-exterior forest is balanced; that is, every $\mathcal{Y}$-exterior forest admits a vector $\alpha = (\alpha_R, \alpha_D) \in \mathbb{R}_{\geq 0}^{r+d}$ satisfying the requirements of Definition \ref{balancing}. We will accomplish this by constructing a sequence of reactions which may be executed indefinitively, and then demonstrating that this sequence repeats. We will define $\alpha$ based on a specific repeating portion of this sequence and show that it is balanced.

Let $\mathbf{X} \in \mathbb{Z}_{\geq 0}^m$ denote our initial state. By assumption, there is a state $\mathbf{X}^{1-} \in \mathbb{Z}_{\geq 0}^m$ and a complex $y^{1-} \not\in\mathcal{Y}$ such that (i) $\mathbf{X} \leadsto \mathbf{X}^{1-}$, (ii) $y^{1-}$ is charged at $\mathbf{X}^{1-}$, and (iii) no complex $y \not\in \mathcal{Y}$ is charged at any state along the sequence of reactions from $\mathbf{X} \leadsto \mathbf{X}^{1-}$. That is, $y^{1-}$ is the first $\mathcal{Y}$-exterior complex which becomes charged as a result of the reaction sequence. (If $y^{1-}$ is charged at $\mathbf{X}$ then the sequence of reactions is empty.) By construction, there is a unique path in the exterior forest from $y^{1-}$ to $\mathcal{Y}$. Let $y^{1+} \in \mathcal{Y}$ denote the complex at the end of this path and $\mathbf{X}^{1+} \in \mathbb{Z}_{\geq 0}^m$ denote the state obtained by the sequential occurrence of the true reactions in the path (i.e. include reactions in $\mathcal{R}_F$ but exclude domination relations $\mathcal{D}_F$). Note that (i) $\mathbf{X}^{1-} \leadsto \mathbf{X}^{1+}$, and (ii) $y^{1+}$ is charged at $\mathbf{X}^{1+}$.

We now iterate this procedure for $i=2,3,4,\ldots,$ starting from the state $\mathbf{X}^{(i-1)+}$ rather than $\mathbf{X}$. This generates the following sequence of transitions, which may be continued indefinitely by the recurrence assumption:
\begin{equation}
\label{stuff1}
\mathbf{X} \leadsto \mathbf{X}^{1-} \leadsto \mathbf{X}^{1+} \leadsto \mathbf{X}^{2-} \leadsto \mathbf{X}^{2+} \leadsto \cdots
\end{equation}
Since the CRN is subconservative, we have that there is a finite number of accessible states (Theorem 1, \cite{DBLP:conf/ac/MemmiR75}). It follows that there is a state in $\{ \mathbf{X}^{1-}, \mathbf{X}^{2-}, \ldots \}$ which is repeated. We let $n_1$ and $n_2$ where $0 < n_1 < n_2$ denote the first and second indices for the set $\{ \mathbf{X}^{1-}, \mathbf{X}^{2-}, \ldots \}$ such that $\mathbf{X}^{n_1-} = \mathbf{X}^{n_2-}$. This gives the following subsequence of (\ref{stuff1})
\begin{equation}
\label{stuff2}
\mathbf{X}^{n_1-} \leadsto \mathbf{X}^{n_1+} \leadsto \cdots \leadsto \mathbf{X}^{(n_2-1)+} \leadsto \mathbf{X}^{n_2-}.
\end{equation}
Since $\mathbf{X}^{n_1-} = \mathbf{X}^{n_2-}$, (\ref{stuff2}) defines a sequence of reactions which can be repeated indefinitely.

We now define the vector $\alpha = (\alpha_R,\alpha_D) \in \mathbb{Z}_{\geq 0}^{r+d}$ in the following way: (i) $\alpha_R$ consists of the counts of the reactions in the sequence of reactions in (\ref{stuff2}), and (ii) $\alpha_D$ consists of the counts of the domination relations in the paths taken to construct the reaction sequences in (\ref{stuff2}).

We now show that $\alpha$ if balanced according to Definition \ref{balancing}. It is clear, first of all, that $\alpha$ only has support on $\mathcal{R}_F$ and $\mathcal{D}_F$ so that Condition $1$ is satisfied. In order to show that $\alpha_R \in \mbox{ker}(\Gamma)$, we note from Eqn.\ (\ref{markov}) of the main text, and the definition of $\alpha_R$, that
\[\mathbf{X}^{n_2-} = \mathbf{X}^{n_1-} + \Gamma \alpha_R \; \; \; \Longrightarrow \; \; \; \mathbf{0} = \Gamma \alpha_R.\]
It follows that $\alpha_R \in \mbox{ker}(\Gamma)$ and therefore $\alpha$ satisfies Condition $2$ of Definition \ref{balancing}. To verify Condition $3$, we note that, since $y^{i-}$ is always chosen to be the first complex exterior to $\mathcal{Y}$ which becomes charged, the only contribution to $\alpha$ from the nonterminal component comes from the segments corresponding to $\mathbf{X}^{i-} \leadsto \mathbf{X}^{i+}$, i.e. the paths from $y^{i-}$ to $\mathcal{Y}$. It follows that, at every complex exterior to $\mathcal{Y}$, the count of the reaction out is at least as great as the sum of the reactions in, and $\alpha$ therefore satisfies Condition $3$ of Definition \ref{balancing}. The result is therefore shown.
\end{proof}

\section{Recurrence Properties of SLCs}
\label{appendixc}


We define the following, which extends Definition \ref{recurrence}.

\begin{definition}
Consider a CRN on a discrete state space. An SLC $W \in \mathcal{W}$ is said to be \textbf{recurrent} from state $\mathbf{X} \in \mathbb{Z}_{\geq 0}^m$ if every $y \in W$ is recurrent from $\mathbf{X}$; otherwise, we will say $W$ is \textbf{transient} from $\mathbf{X}$.
\end{definition}

\noindent In other words, an SLC is recurrent if every complex $y$ in the SLC is recurrent. We now consider how the transient and recurrence of complexes may be distributed throughout a CRN.


\begin{lemma}
\label{lemma10}
Consider a CRN on a discrete state space. Then a complex $y \in W$ is recurrent if and only if the SLC $W \in \mathcal{W}$ is recurrent.
\end{lemma}

\begin{proof}
Let $W \in \mathcal{W}$ denote an SLC of a CRN. Suppose $y \in W$ is recurrent from $\mathbf{X} \in \mathbb{Z}_{\geq 0}^m$ and $y' \in W$. From the recurrence of $y$, it follows that, for every $\mathbf{Y} \in \mathbb{Z}_{\geq 0}^m$ such that $\mathbf{X} \leadsto \mathbf{Y}$, there is a $\mathbf{Z} \in \mathbb{Z}_{\geq 0}^m$ such that $\mathbf{Y} \leadsto \mathbf{Z}$ and $y$ is charged at $\mathbf{Z}$. Since $y$ and $y'$ belong to the same linkage class, it follows that there is a path from $y$ to $y'$. It follows that there is a $\mathbf{W} \in \mathbb{Z}_{\geq 0}^m$ such that $\mathbf{Z} \leadsto \mathbf{W}$ and $y'$ is charged at $\mathbf{W}$. We have used the observation that reactions in a path in the reaction graph may occur in sequence since each reaction necessarily produces sufficient molecularity for the next reaction to proceed. We therefore have that $\mathbf{X} \leadsto \mathbf{Y} \leadsto \mathbf{Z} \leadsto \mathbf{W}$ so that, for every $\mathbf{Y} \in \mathbb{Z}_{\geq 0}^m$, such that $\mathbf{X} \leadsto \mathbf{Y}$, there is a $\mathbf{W} \in \mathbb{Z}_{\geq 0}^m$ such that $\mathbf{Y} \leadsto \mathbf{W}$ and $y'$ is charged at $\mathbf{W}$. It follows that $y'$ is recurrent from $\mathbf{X}$, and we are done.
\end{proof}


\noindent This result shows that complex and SLCs recurrent is equivalent in the sense that we may not have one without the other. We may further relate recurrence to the reaction graph of CRN with the following.

\begin{lemma}
\label{lemma1}
Consider a CRN on a discrete state space and a dom-CRN. Suppose there is a path in the maximal dom-CRN from a complex $y \in W$ to a complex $y' \in W'$ where $W,W' \in \mathcal{W}$ are two SLCs of the CRN. Then the following hold:
\begin{enumerate}
\item
If $W$ is recurrent from $\mathbf{X} \in \mathbb{Z}_{\geq 0}^m$, then $W'$ is recurrent from $\mathbf{X}$.
\item
If $W'$ is transient from $\mathbf{X} \in \mathbb{Z}_{\geq 0}^m$, then $W$ is transient from $\mathbf{X}$.
\item
The set of recurrent complexes is an absorbing complex set of the maximal dom-CRN consisting of the union of SLCs.
\end{enumerate}
\end{lemma}

\begin{proof}
Let $W, W' \in \mathcal{W}$ denote an SLC of a CRN and suppose $W$ is recurrent from $\mathbf{X} \in \mathbb{Z}_{\geq 0}^m$. Suppose there is a path in the dom-CRN from a complex $y \in W$ to a complex $y' \in W'$. It follows that, for every $\mathbf{Y} \in \mathbb{Z}_{\geq 0}^m$ such that $\mathbf{X} \leadsto \mathbf{Y}$, there is a $\mathbf{Z} \in \mathbb{Z}_{\geq 0}^m$ such that $\mathbf{Y} \leadsto \mathbf{Z}$ and $y$ is charged at $\mathbf{Z}$.

Now consider the path in the dom-CRN from $y$ to $y'$. The path may be composed of reactions in $\mathcal{R}$ or $\mathcal{D}$. We have that, for any sequence of reactions in the path which are only from $\mathcal{R}$, if the path starts with a recurrent complex, recurrence is transferred to every complex in the path, including the final one. This can be realized by noting that the occurrence of each reaction necessarily confers sufficient molecularity for the next reaction in the path to take place. Also notice that, for any reaction in $\mathcal{D}$, say $y^{*} \stackrel{D}{\longrightarrow} y^{**}$ we have $y_{k}^{**} \leq y_{k}^{*}$ for all $k \in \{1, \ldots, m\}$. It follows that, if $y^{*}$ is charged at a state, then $y^{**}$ is charged at the state. Combining these two results, we have that there is a $\mathbf{W} \in \mathbb{Z}_{\geq 0}^m$ such that $\mathbf{Z} \leadsto \mathbf{W}$ and $y'$ is charged at $\mathbf{W}$. It follows that $y'$ is recurrent from $\mathbf{X}$. It follows that $W'$ is recurrent from Lemma \ref{lemma10}.

This proves Claim $1$ and, since Claim $2$ is the contrapositive of Claim $1$, this is also shown. Claim $3$ follows by noting that, if $\mathcal{Y} \subseteq \mathcal{C}$ is the set of recurrent complexes but is not an absorbing complex set in the maximal dom-CRN then there is a reaction $(y_i,y_j) \in \ \mathcal{R} \cup \mathcal{D}$ such that $y_i \in \mathcal{Y}$ and $y_j \not\in \mathcal{Y}$. It then follows from Claim $1$ that $y_j$ is recurrent, which contradicts the construction of $\mathcal{Y}$. Since the set of recurrent complexes must consist of the union of SLCs by Lemma \ref{lemma10}, the result is shown.
\end{proof}

\noindent This result gives restrictions on the distribution of transient and recurrent complexes and SLCs within a CRN. Recurrence travels with the direction of the paths in the dom-CRN while transience travels against the direction of these paths. Claim $3$ furthermore suggests that absorbing complex sets are the correct object of study when considering recurrence and transience in discrete state space CRNs. Consider the following example.

\begin{example}
Consider the CRN structure contained in Figure \ref{figure1}, where the boxes represent SLCs. We write $W_j \to W_i$ if there is a path from some $y_i \in W_i$ to $y_j \in W_j$ in the CRN, and we write $W_j \stackrel{D}{\longrightarrow} W_i$ if $y_i \leq y_j$ for some $y_i \in W_i$ and $y_j \in W_j$. In (b) and (c), potential patterns for recurrent SLCs consistent with Lemma \ref{lemma1} are highlighted. If $W_6$ is recurrent, $W_5, W_7,$ and $W_8$ must be recurrent as well (red). If $W_2$ is recurrent, $W_4$ and $W_5$ must be recurrent as well (green). Note that in (c) this recurrence implication flows through the domination relationship. Note also that Lemma \ref{terminallemma} guarantees that a subconservative CRN may not contain any cycles in the representation Figure \ref{figure1}.
\end{example}

\begin{figure}[h]
\centering
\includegraphics[width=12cm]{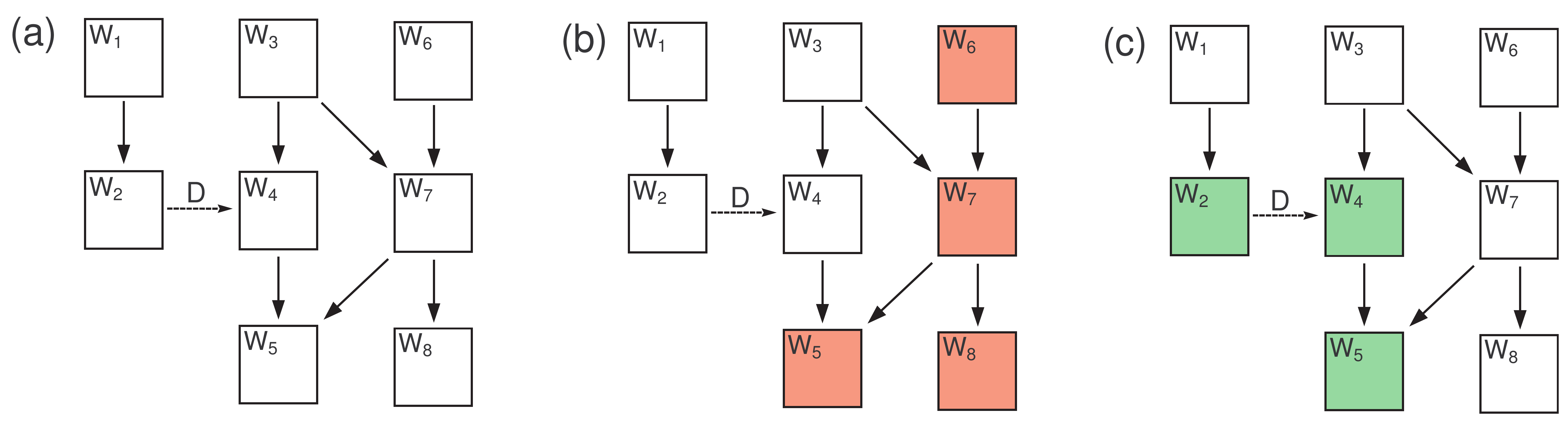}
\caption{\small In (a), a schematic diagram of the domination structure for a CRN with $8$ SLCs. In (b) and (c), two possible configurations of recurrent SLCs (red and green, respectively) are presented which are consistent with Lemma \ref{lemma1}.}
\label{figure1}
\end{figure}


\begin{remark}
For subconservative CRNs, another way to interpret Lemma \ref{lemma1} is by introducing the order operation $\preceq_D$ on the SLCs of a CRN, where $W \preceq_D W'$ if there are $y \in W$ and $y' \in W'$ such that $y \leq y'$. In \cite{Brijder}, R. Brijder showed that the transition closure of this operator is a partial order on the SLCs of a subconservative CRN.

We extend this slightly by defining the relation $\preceq_*$ to be such that $W \preceq_* W'$ if either $W \preceq_D W'$ or there is a path from a complex $y' \in W'$ to a complex $y \in W$ in the CRN, and then let $\preceq$ denote the transitive closure of $\preceq_*$. Since the relation $\preceq$ corresponds to path-connectedness in the dom-CRN, Lemma \ref{terminallemma} is equivalent to the property of $\preceq$ being a partial order on the SLCs of a subconservative CRN. In this interpretation, the minimal SLCs of the CRN under the partial order $\preceq$ correspond to the terminal SLCs of the dom-CRN. We may then interpret Lemma \ref{lemma1} as stating that, for a subconservative CRN with SLCs $W,W' \in \mathcal{W}$ such that $W' \preceq W$, (a) if $W$ is recurrent from $\mathbf{X} \in \mathbb{Z}_{\geq 0}^m$, then $W'$ is recurrent from $\mathbf{X}$, and (b) if $W'$ is transient from $\mathbf{X} \in \mathbb{Z}_{\geq 0}^m$, then $W$ is transient from $\mathbf{X}$. That is, relative to the partial order $\preceq$ on SLCs of a subconservative CRN, recurrence flows downward while transience flows upward.
\end{remark}

\section{Connection with Petri Nets}
\label{appendixd}




Petri nets form a well-studied model of concurrent computation, see, e.g., \cite{PetriNet/review/Pet1977,DBLP:conf/ac/1996petri1}. Petri nets are essentially\footnote{The word ``essentially'' is due to the fact that, unlike CRNs, Petri nets usually have a fixed initial marking $M$. However, this difference is irrelevant for this paper.} equivalent to CRNs on discrete state spaces.  As a consequence, results concerning CRNs on discrete state spaces can be equivalently stated in terms of Petri nets and vice versa.

In a Petri net (without initial marking) $N = (P,T,F)$, species are called \emph{places} (i.e., $P = \mathcal{S}$), reactions are called \emph{transitions} (i.e., $T=\mathcal{R}$), and the stoichiometric coefficients of each reaction is encoded by a function $F: T \to \mathbb{Z}_{\geq 0}^P \times \mathbb{Z}_{\geq 0}^P$. Moreover, molecules are called \emph{tokens} and states $\mathbf{X}$ are called \emph{markings} $M$. Furthermore, the stoichiometric matrix $\Gamma$ is known as the \emph{incidence matrix} of a Petri net, conservation vectors $\mathbf{c} \in \mathbb{Z}_{> 0}^P$ are known as \emph{$P$-invariants}, and vectors $\mathbf{v} \in \mathbb{Z}_{\geq 0}^P$ such that $\mathbf{v} \in \mbox{ker}(\Gamma)$ are known as \emph{$T$-invariants}.

Graphical depictions are different for Petri nets compared to CRNs. For example, consider the following example.

\begin{example}
Reconsider the CRN from Example \ref{example234}, which is represented graphically in Figure~\ref{fig:petri}. In the setting of Petri nets, species/places are denoted by circles and reactions/transitions by boxes. Moreover, the reactants of a reaction/transition are the incoming edges of that transition (including multiplicity as edge labels) and the products of a reaction/transition are the outgoing edges of that transition (including multiplicity as edge labels). The Petri net has the incidence matrix:
\[
\Gamma =
\bordermatrix{
~ & 1 & 2 & 3 \cr
X_1 & -1 & 1 & 1 \cr
X_2 & 1 & -1 & -1
}.
\]
A $P$-invariant is given by the vector $(1,1)$ and $T$-invariants by the vectors $(1,1,0)^T$ and $(1,0,1)^T$.

\begin{figure}
\begin{center}
\scalebox{1.1}{
\begin{tikzpicture}[auto, >=latex']
\tikzstyle{place}=[circle,draw=blue!50,fill=blue!20,thick,
inner sep=0pt,minimum size=6mm]
\tikzstyle{transition}=[rectangle,draw=black!50,fill=black!20,thick,
inner sep=0pt,minimum size=4mm]
\node[place] (xA) {$X_1$};
\node[transition] (b) [below of=xA] {$3$};
\node[place] (xB) [below of=b] {$X_2$};
\node[transition] (a) [left of=b,xshift=-5mm] {$1$};
\node[transition] (c) [right of=b,xshift=5mm] {$2$};
\draw [->] (xB) to [bend right=30] (a);
\draw [->] (xA) to [bend right=30] (a);
\draw [->] (a) to [bend right=30] node [swap] {2}  (xB);
\draw [->] (xB) to (b);
\draw [->] (b) to (xA);
\draw [->] (xB) to [bend right=30] node [swap] {2} (c);
\draw [->] (c) to [bend right=30] (xA);
\draw [->] (c) to [bend right=30] (xB);
\end{tikzpicture}}
\end{center}
\caption{Petri net-style depiction of the CRN of Example \ref{example234}.}
\label{fig:petri}
\end{figure}
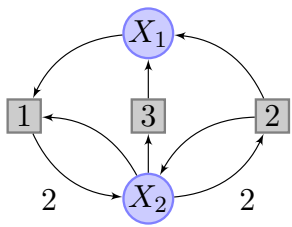
\end{example}
\end{appendices}

\end{document}